\documentclass[11pt]{article}
\usepackage[utf8]{inputenc}

\usepackage{epsfig,epsf,fancybox}
\usepackage{amsmath}
\usepackage{mathrsfs,bbm}
\usepackage{amssymb}
\usepackage{graphicx}
\usepackage{color}
\usepackage{multirow}
\usepackage{paralist}
\usepackage{verbatim}
\usepackage{galois}
\usepackage{algorithm}
\usepackage{algorithmic}
\usepackage{boxedminipage}
\usepackage{booktabs}
\usepackage{accents}
\usepackage{stmaryrd}
\usepackage{subfig}
\usepackage{appendix}
\usepackage{graphicx}
\usepackage{epstopdf}
\usepackage{appendix}
\usepackage{amsthm}
\usepackage{cases}  
\usepackage[top=1in,bottom=1.2in,left=1in,right=1in,xetex]{geometry}
\usepackage{extarrows}

\usepackage{titlesec}
\usepackage{titletoc}

\usepackage[unicode=true,pdfusetitle,
bookmarks=true,bookmarksnumbered=true,bookmarksopen=true,bookmarksopenlevel=3,
breaklinks=false,pdfborder={0 0 1},backref=false,colorlinks=false]
{hyperref}

\numberwithin{equation}{section}

\newtheorem{theorem}{Theorem}[section]
\newtheorem{corollary}[theorem]{Corollary}
\newtheorem{lemma}[theorem]{Lemma}

\newtheorem{condition}{Condition}[section]
\newtheorem{definition}{Definition}[section]

\newtheorem{remark}{Remark}[section]

\newcommand{\f}{\mathscr{F}}
\newcommand{\lr}{\mathcal{L}}
\newcommand{\sr}{\mathcal{S}}
\newcommand{\e}{\mathbb{E}}
\newcommand{\br}{\mathbb{R}}
\newcommand{\pr}{\mathcal{P}}
\newcommand{\dd}{\partial}

\newcommand{\brn}{{\mathbb{R}^n}}

\newcommand{\brd}{{\mathbb{R}^d}}
\newcommand{\de}{\Delta}

\newcommand{\hv}{\widehat{v}}

\newcommand{\tx}{\widetilde{X}}

\newcommand{\tz}{\widetilde{Z}}

\newcommand{\argmin}{\mathop{\rm argmin}}

\allowdisplaybreaks[2]

\title{On Mean Field Monotonicity Conditions from Control Theoretical Perspective}

\usepackage{authblk}

\author[a]{Alain Bensoussan\footnote{E-mail: axb046100@utdallas.edu}}
\author[b]{Ziyu Huang\footnote{E-mail: zyhuang19@fudan.edu.cn}}
\author[c]{Shanjian Tang\footnote{E-mail: sjtang@fudan.edu.cn}}
\author[b]{Sheung Chi Phillip Yam\footnote{E-mail: scpyam@sta.cuhk.edu.hk}}

\affil[a]{\small \it International Center for Decision and Risk Analysis, Naveen Jindal School of Management, University of Texas at Dallas, Dallas, Texas, USA}
\affil[b]{\small \it Department of Statistics, The Chinese University of Hong Kong, Shatin, N.T., Hong Kong SAR}
\affil[c]{\small \it Department of Finance and Control Sciences, School of Mathematical Sciences, Fudan University, Shanghai 200433 , PRC}

\begin{document}

\maketitle


\medskip

\begin{abstract}
\noindent In this article, from the viewpoint of control theory, we discuss the relationships among the commonly used monotonicity conditions that ensure the well-posedness of the solutions arising from problems of mean field games (MFGs) and mean field type control (MFTC). We first introduce the well-posedness of general forward-backward stochastic differential equations (FBSDEs) defined on some suitably chosen Hilbert spaces under the $\beta$-monotonicity. We then propose a monotonicity condition for the MFG, namely partitioning  the running cost functional into two parts, so that both parts still depend on the control and the state distribution, yet one satisfies a strong convexity and a small mean field effect condition, while the other has a newly introduced displacement quasi-monotonicity. To the best of our knowledge, the latter quasi type condition has not yet been discussed in the contemporary literature, and it can be considered as a bit more general monotonicity condition than those commonly used. Besides, for the MFG,  we show that convexity and small mean field effect condition for the first part of running cost functional and the quasi-monotonicity condition for the second part together imply the $\beta$-monotonicity and thus the well-posedness  for the associated  FBSDEs. For the MFTC problem, we show that the  $\beta$-monotonicity for the corresponding FBSDEs is simply  the convexity assumption on the cost functional. Finally, we consider a more general setting where the drift functional is allowed to be non-linear for both MFG and  MFTC problems. \\

\noindent{\textbf{Keywords:}} Mean field games; Mean field type control problems; Forward-backward stochastic differential equations; Monotonicity conditions; $\beta$-Monotonicity; Small mean field effect; Displacement quasi-monotonicity; Generic drift functions.\\

\noindent {\bf Mathematics Subject Classification (2020):} 60H30; 60H10; 93E20; 35R15. 

\end{abstract}

\section{Introduction}

Mean field games (MFGs) and mean field type control (MFTC) problems have been widely studied in recent years. For each of them, the underlying controlled dynamical system involves the probability distribution of the state, in addition to the state and the control. A MFG is essentially a fixed point problem which was first proposed by Lasry and Lions in a series of papers \cite{CP,JM1,JM2,JM3} and also independently by Huang, Caines and Malhamé \cite{HM1,HM}. In a MFG, the coefficients of the controlled system are affected by the equilibrium probability distribution of the state of the overall population. In contrast,  a MFTC problem is a McKean--Vlasov control problem, the system of which depends on the distribution of the current controlled state. There are numerous works in various settings in this area. For PDE approches to forward-backward system for MFGs, we refer to Bensoussan--Frehse--Yam \cite{AB_book}, Gomes--Pimentel--Voskanyan \cite{GDA}, Graber--M\'{e}sz\'{a}ros \cite{GM}, Huang--Tang \cite{HZ} and Porretta \cite{PA}. For the master equation analytical methods to MFGs, we refer to Cardaliaguet--Cirant--Porretta \cite{CP1}, Cardaliaguet--Delarue--Lasry--Lions \cite{CDLL}, Gangbo--M\'esz\'aros--Mou--Zhang \cite{GW} and Mou--Zhang \cite{Anti_mono}. For probabilistic approaches to MFGs, we refer to Ahuja--Ren--Yang \cite{SA1}, Bensoussan--Tai--Wong--Yam \cite{AB9'}, Bensoussan--Wong--Yam--Yuan \cite{AB10'}, Buckdahn--Li--Peng-- Rainer \cite{BR}, Carmona--Delarue \cite{book_mfg}, Chassagneux--Crisan--Delarue \cite{CJF} and Huang--Tang \cite{HZ1}. For probabilistic approaches to MFTC problems, we refer to Buckdahn--Li--Peng--Rainer \cite{BR}, Cardaliaguet--Delarue--Lasry--Lions \cite{CDLL}, Carmona--Delarue \cite{CR,book_mfg} and Chassagneux--Crisan--Delarue \cite{CJF}. For the dynamic programming principle and HJB equation of McKean--Vlasov control problem, we refer to Djete--Possamai--Tan \cite{DMF} and Pham--Wei \cite{PH}. For the lifting method and a Hilbert space approach for MFTC problem, we refer to Bensoussan--Graber--Yam \cite{AB7}, Bensoussan--Huang--Yam \cite{AB10,AB6,AB8}, Bensoussan--Tai--Yam \cite{AB5} and Bensoussan--Yam \cite{AB}. The MFTC problem is also studied together with the so called ``potential mean field game", which is first observed by Lasry and Lions in \cite{JM3} (although under a different name), and then subsequently well-studied in the literature; see \cite{Briani2017,Cardaliaguet2015,ORRIERI20191868} as examples. 

In this article, we adopt a stochastic control method to discuss monotonicity conditions to solve MFG and MFTC problems with control dependent diffusion functions which can be degenerate. Let $(\Omega,\f,\{\f_t,0\le t\le T\},\mathbb{P})$ be a complete filtered probability space (with the filtration being  augmented by all the $\mathbb{P}$-null sets)  on which an $n$-dimensional Brownian motion $\{B_t,{0\le t\le T}\}$ is defined and is $\f_t$-adapted. We denote by $\mathcal{P}_{2}(\brn)$ the space of all probability measures with finite second-order moments on $\brn$,  equipped with the 2-Wasserstein metric $W_2$. Given the functional coefficients:
\begin{align*}
	&b:[0,T]\times\brn\times\mathcal{P}_{2}(\brn)\times \brd\to \brn,\quad \sigma:[0,T]\times \brn\times\mathcal{P}_{2}(\brn)\times \brd\to \br^{n\times n},\\
	&f:[0,T]\times\brn\times\mathcal{P}_{2}(\brn)\times \brd\to \br,\quad g:\brn\times\mathcal{P}_{2}(\brn)\to \br,
\end{align*}
for $(t,\mu)\in[0,T]\times\pr_2(\brn)$, we choose a random vector $\xi\in L_{\f_t}^2$ independent of the Brownian motion $\mathcal{W}_s^t:=\{B_s-B_t, t\le s\le T\}$ such that $\lr(\xi)=\mu$, then, the limiting problem of MFG can be formulated in the following probabilistic way:
\begin{equation}\label{intro_MFG}
	\left\{
	\begin{aligned}
		&\hv_\cdot\in\argmin_{v_\cdot\in\mathcal{L}_{\mathscr{F}}^2(t,T)}J\left(v_\cdot;\widehat{m}_s,t\le s\le T\right):=\e\left[\int_t^T f\left(s,X^v_s,\widehat{m}_s,v_s\right)dt+g\left(X^v_T,\widehat{m}_T\right)\right],\\
		&X^v_s=\xi+\int_t^sb\left(r,X^v_r,\widehat{m}_r,v_r\right)dr+\int_t^s\sigma\left(r,X^v_r,\widehat{m}_r,v_r\right)dB_r,\\
		&\widehat{m}_s:=\mathcal{L}\left(X^{\hv}_s\right),\quad  s\in[t,T].
	\end{aligned}
	\right.
\end{equation}
From the  stochastic maximum principle \cite{AB11,book_mfg,HZ1,MR1696772}, the solution of MFG \eqref{intro_MFG} (if exits) gives the well-posedness of the following system of forward-backward stochastic differential equations (FBSDEs):
\begin{equation}\label{intro_2}
	\left\{	
	   \begin{aligned}
		   &X_s=\xi+\int_t^s D_pH\left(r,X_r,\lr(X_r),P_r,Q_r\right)dr+\int_t^s D_qH(X_r,\lr(r,X_r),P_r,Q_r)dB_r,\\
		   &P_s=D_x g\left(X_T,\lr(X_T)\right)+\int_s^T D_x H(r,X_r,\lr(X_r),P_r,Q_r)dr-\int_s^T Q_rdB_r,\quad s\in[t,T].
	   \end{aligned}
	\right.
\end{equation}
Here, the Hamiltonian $H:[0,T]\times\brn\times\pr_2(\brn)\times\brn\times\br^{n\times n}\to\br$ is defined as
\begin{align}
		&H (s,x,m,p,q):=\inf_{v\in\brd} L\left(s,x,m,v,p,q\right), \notag 
  \\
		&L(s,x,m,v,p,q):=p^\top  b(s,x,m,v)+\sum_{j=1}^n \left(q^j\right)^\top \sigma^j(s,x,m,v)+f(s,x,m,v). \notag 
\end{align}
Reversely,  the well-posedness of FBSDEs \eqref{intro_2}  also gives a solution to MFG \eqref{intro_MFG}. Another main concern of this article is to study the following MFTC problem in a more probabilistic favor:
\begin{equation}\label{intro_MFTC}
	\left\{
	\begin{aligned}
		&\inf_{v_\cdot\in \lr^2_{\f}(t,T)} \ J(v_\cdot):=\left[\e \int_{t}^{T}f\left(s,X^v_s,\lr\left(X^v_s\right),v_s\right)ds+g\left(X^v_T,\lr\left(X^v_T\right)\right)\right],\\
		&X^v_s=\xi+\int_t^s b\left(r, X^v_r,\lr\left(X^v_r\right),v_r\right)dr +\int_t^s \sigma\left(r, X^v_r,\lr\left(X^v_r\right),v_r\right)dB_r,\quad s\in[t,T].
	\end{aligned}
	\right.
\end{equation}
From the stochastic maximum principle for McKean--Vlasov type control problems \cite{AB8,CR}, the optimal control of the MFTC problem \eqref{intro_MFTC} is associated with the following FBSDEs:
\begin{equation}\label{FB:4}
	\left\{
	\begin{aligned}
		X_s=\ & \xi +\int_t^s D_p H(r,X_r,\lr(X_r),P_r,Q_r)dr +\int_t^s  D_{q}H(r,X_r,\lr(X_r),P_r,Q_r)dB_r,\\
		P_s=\ & D_xg\left(X_T,\lr(X_T)\right) +\widetilde{\e}\left[D_y\frac{d g}{d\nu}\left(\widetilde{X_T},\lr(X_T)\right)(X_T)\right]\\
		& +\int_s^T\bigg\{D_x H(r,X_r,\lr(X_r),P_r,Q_r)+\widetilde{\e}\left[D_y \frac{d H}{d\nu}\left(r,\widetilde{X_r},\lr(X_r),\widetilde{P_r},\widetilde{Q_r}\right)(X_r) \right]\bigg\}dr\\
		& - \int_s^T Q_rdB_r,\quad s\in[t,T],
	\end{aligned}
	\right.
\end{equation}
where the processes $\widetilde{X_s}$, $\widetilde{P_s}$ and $\widetilde{Q_s}$ are respectively  independent copies of $X_s$, $P_s$ and $Q_s$, and $\frac{dH}{d\nu}$ is the linear functional derivative of $H$ with respect to the measure variable (see Section~\ref{sec:notation} below). The backward equation in FBSDEs \eqref{FB:4} is different from that of FBSDEs \eqref{intro_2} by including the terms involving $D_y\frac{dH}{d\nu}$ and $D_y\frac{dg}{d\nu}$; this is because the state process of the MFTC problem \eqref{intro_MFTC} depends simultaneously on the distribution of the current controlled state, while the state process of MFG \eqref{intro_MFG} depends on the the equilibrium distribution from the population. The main result of this work is to give the well-posedness of FBSDEs \eqref{intro_2} under different kinds of monotonicity conditions on $f$, and also give the well-posedness of FBSDEs \eqref{FB:4} under a convexity assumption on $f$ in $(x,v)\in\brn\times\brd$ and $m\in \pr_2(\brn)$. Our  approach is more aligned with the traditional stochastic control method, which is significantly different from the recently advocated analytical approach to MFGs via master equation or PDE-based HJB-FP equations in the existing literature. The control theoretic perspective allows us to cope with the state and control linear diffusion $\sigma$, which is  unbounded and possibly degenerate. We allow our running cost functional $f$ to be non-separable in $v$ and $m$ with a quadratic growth, and to satisfy a strong convexity or even a more general displacement quasi-monotonicity. Moreover, we  also solve both MFG \eqref{intro_MFG} and the MFTC problem \eqref{intro_MFTC} with generic drift functionals $b$. We also require less regularity of the cost functions than  the existing literature  to solve the original mean field problems.

To study the well-posedness of FBSDEs \eqref{intro_2} and \eqref{FB:4}, we first establish the solvability of a general system of FBSDEs defined on Hilbert space (see FBSDEs \eqref{FB:11}) under our $\beta$-monotonicity (see Condition \ref{Condition_mono}), and then apply this general well-posedness result for FBSDEs \eqref{intro_2} and \eqref{FB:4}, both of which can be viewed as particular cases of FBSDEs \eqref{FB:11}. The monotonicity condition for fully coupled FBSDEs was first introduced by Hu--Peng \cite{YH2} and Peng--Wu \cite{SP} together with the continuation method. This monotonicity condition was then used as a convexity condition in \cite{CR} for FBSDEs arising from MFTC problem, and it was also used as a weak monotonicity condition (displacement monotonicity condition) in \cite{SA1,HZ1} for FBSDEs arising from MFG with a common noise. Compared with the usual monotonicity condition for FBSDEs, our $\beta$-monotonicity allows us to include more general situations by choosing a suitable candidate of $\beta$, for instance, it can be used to establish the well-posedness of the  Jacobian and Hessian flows for FBSDEs \eqref{intro_2} and \eqref{FB:4} by choosing suitable maps of $\beta$; see our previous work \cite{AB10,AB11}. In this work, we further extend our previous results in \cite{AB10,AB11} by relaxing assumptions on $f$, and also including the generic drift functionals. For the FBSDEs \eqref{intro_2} arising from MFG \eqref{intro_MFG}, we propose a new monotonicity condition on $f$ (see Condition~\ref{assumption_mono_1}): the functional $f$ is assumed to be divided into two parts, both dependent on $(s,x,m,v)$, with one part satisfying a strong convexity condition in $(x,v)$, and the other part satisfying the newly proposed displacement quasi-monotonicity. To the best of our knowledge, this condition is brand-new in mean field theory, and it can be considered as a more general monotonicity condition; for example, it can include as interesting special cases the one with separability and displacement monotonicity proposed in \cite{SA1} and also the ``strong convexity and small mean field effect" condition proposed in our previous work \cite{AB11}, and it also overlaps with monotonicity conditions proposed in \cite{GW,GM,Anti_mono} from an analytical viewpoint. We show that, the convexity and the quasi-monotonicity conditions in our Condition~\ref{assumption_mono_1} can imply the corresponding $\beta$-monotonicity for FBSDEs \eqref{intro_2}, and therefore ensure the well-posedness of the FBSDEs. As two particular cases of our Condition~\ref{assumption_mono_1}, the classical displacement monotonicity and the small mean field effect can both be viewed as a condition to ensure the $\beta$-monotonicity. For the FBSDEs \eqref{FB:4} arising from the MFTC problem \eqref{intro_MFTC}, we need the cost functional $f$ to be jointly convex in $(x,v)$ and also convex in $m$; see Assumption (B3) in Section~\ref{sec:MFTC}. We show that the corresponding $\beta$-monotonicity for FBSDEs \eqref{FB:4} is exactly this notion of convexity assumption. Moreover, for both FBSDEs \eqref{intro_2} and \eqref{FB:4}, we also study the case when the drift functional $b$ can be non-linear in $x$, $v$ and $m$; see Assumptions (A1') and (B1') in Section~\ref{sec:generic}. We can prove that, in these cases, the coefficients satisfy the $\beta$-monotonicity; therefore, our well-posedness result for general FBSDEs defined on Hilbert space under $\beta$-monotonicity is feasible. 

MFG \eqref{intro_MFG} is usually associated with a mean field master equation, and the MFTC problem \eqref{intro_MFTC} is usually associated with a Bellman euqation. In \cite{AB10}, we establish the classical solutions of the Bellman equation corresponding to the MFTC problem; and in \cite{AB11}, we also establish the classical solution of the MFG master equation. We can see that, to obtain a classical solution of a master equation or HJB equation requires more restrictive assumptions and higher regularity on coefficients than just to obtain an equilibrium solution of MFG or an optimal control for the MFTC problem. Therefore, we study MFG \eqref{intro_MFG} and the MFTC problem \eqref{intro_MFTC} by a stochastic control method, which can include more cases. To make the control perspective more clear, we prefer to give conditions directly on coefficients $b$, $\sigma$, $f$ and $g$, rather than give conditions on the Hamiltonian functional $H$ or other feedback maps of $m$. We also refer to other discussions on monotonicity conditions for MFG from an analytical perspective; see \cite{GW,GM,Anti_mono}. 

The rest of the paper is organized as follows. In Section~\ref{sec:Hilbert}, we establish the well-posedness of general FBSDEs defined on Hilbert space under the $\beta$-monotonicity. In Section~\ref{sec:MFG}, we give a sufficient condition of maximum principle for MFG \eqref{intro_MFG}, and propose different monotonicity conditions on $f$ to ensure the well-posedness of FBSDEs \eqref{intro_2}. Section~\ref{sec:MFTC} gives a sufficient condition of maximum principle for the MFTC problem \eqref{intro_MFTC}, and we shall show that the convexity of $f$ can ensure the well-posedness of FBSDEs \eqref{FB:4}. In Section~\ref{sec:generic}, we study the case when the drift functional $b$ is non-linear and generic for both MFG \eqref{intro_MFG} and the MFTC problem \eqref{intro_MFTC}.

\subsection{Notations}\label{sec:notation}

For any $X\in L^2(\Omega,\f,\mathbb{P};\brn)$, we denote by $\lr(X)$ its law and by $\|X\|_2$ the $L^2$-norm. For every $t\in[0,T]$, we  denote by $L^2_{\f_t}$ the set of all $\f_t$-measurable square-integrable $\brn$-valued random vectors, and denote by $\lr^2_{\f}(0,T)$  the set of all $\f_t$-progressively-measurable $\brn$-valued processes $\alpha_\cdot=\{\alpha_t,\ 0\le t\le T\}$ such that $\e\left[\int_0^T |\alpha_t|^2dt\right]<+\infty$. We  denote by $\mathcal{S}^2_{\f}(0,T)$ the set of all $\f_t$-progressively-measurable $\brn$-valued processes $\alpha_\cdot=\{\alpha_t,\ 0\le t\le T\}$ such that $\e\left[\sup_{0\le t\le T} |\alpha_t|^2\right]<+\infty$. We denote by $\mathcal{P}_{2}(\brn)$ the space of all probability measures of finite second-order moments on $\brn$, equipped with the 2-Wasserstein metric: $W_2\left(m,m'\right):=\inf_{\pi\in\Pi\left(m,m'\right)}\sqrt{\int_{\brn\times\brn}\left|x-x'\right|^2\pi\left(dx,dx'\right)}$, where $\Pi\left(m,m'\right)$ is the set of joint probability measures with respective marginals $m$ and $m'$. We denote by $\delta_0$ the point mass distribution of the random variable $\xi$ such that $\mathbb{P}(\xi=\mathbf{0})=1$. Also see \cite{AL} for more results on Wasserstein metric space.

The linear functional derivative of a functional $k(\cdot):\mathcal{P}_{2}(\brn)\to\br$ at $m\in \mathcal{P}_{2}(\brn)$ is another functional $\mathcal{P}_{2}(\brn)\times \brn\ni(m,y)\mapsto\dfrac{dk}{d\nu}(m)(y)$,  being jointly continuous and satisfying $\int_{\brn}\Big|\dfrac{dk}{d\nu}(m)(y)\Big|^{2}dm(y)\leq c(m)$ for some positive constant $c(m)$ which is bounded on any bounded subsets of $\pr_2(\brn)$, such that 
\begin{equation*}
    \lim_{\epsilon\to0}\dfrac{k((1-\epsilon)m+\epsilon m')-k(m)}{\epsilon}=\int_\brn\dfrac{dk}{d\nu}(m)(y)\left(dm'(y)-dm(y)\right), \quad \forall m'\in\mathcal{P}_{2}(\brn);
\end{equation*}
we refer the reader to \cite{AB,book_mfg} for more details about the notion of linear functional derivatives. In particular, the linear functional derivatives in $\pr_2(\brn)$ are connected to the G\^ateaux derivatives in $L^2(\Omega,\f,\mathbb{P};\brn)$ in the following way. For a linearly functional differentiable functional $k:\pr_2(\brn)\to\br$ such that the derivative $D_y\frac{d k}{d\nu}(\mu)(y)$ is jointly continuous in $(\mu,y)$ and $D_y\frac{d k}{d\nu}(\mu)(y)\le c(\mu)(1+|y|)$ for $(\mu,y)\in\pr_2(\brn)\times\brn$,  the functional $K(X):=k(\lr(X)), \  X\in L^2(\Omega,\f,\mathbb{P};\brn)$ has the following G\^ateaux derivative:
\begin{align}\label{lem01_1}
	D_X K(X)(\omega)=D_y\frac{d k}{d\nu}(\lr(X))(X(\omega)). 
\end{align}
Furthermore, if $k$ is twice linearly functional differentiable, then the functional $K$ is also twice G\^ateaux differentiable, and the G\^ateaux derivative at $X$ along a direction $Z\in L^2(\Omega,\f,\mathbb{P};\brn)$ is
\begin{align*}
    D_X^2 K(X)\left(Z\right)=\left(D_y^2\frac{d k}{d\nu}(\lr(X))(X)\right)^\top Z+\widetilde{\e}\left[ \left(D_{y'}D_y\frac{d^2 k}{d\nu}(\lr(X))\left(X,\tx\right)\right)^\top \tz\right]. 
\end{align*}
Here and in the following of the paper, for any random variable $\xi$, we write $\widetilde{\xi}$ for its independent copy, and $\widetilde{\e}[\widetilde{\xi}]$ for the corresponding expectation.  

For convenience, in this article, we write $f|_{a}^b:=f(b)-f(a)$ for the difference of a  functional $f$ between two points $b$ and $a$. 

\section{Monotonicity Condition for FBSDEs on Hilbert space}\label{sec:Hilbert}

We here give the monotonicity condition to ensure the well-posedness of general FBSDEs defined on Hilbert space, which will be used to give the respective monotonicity conditions so as to ensure the well-posedness of FBSDEs \eqref{intro_2} and \eqref{FB:4} in the following sections. We consider the following FBSDEs defined on Hilbert space of $L^2(\Omega,\f,\mathbb{P};\brn)$: for an initial $(t,\xi)\in[0,T]\times  L_{\f_t}^2$,
\begin{equation}\label{FB:11}
	\left\{
	\begin{aligned}
		&X_s=\xi+\int_t^s \mathbf{B} (r,X_r,P_r,Q_r)dr+\int_t^s \mathbf{A}(r,X_r,P_r,Q_r)dB_r,\\
		&P_s=\mathbf{G}(X_T)-\int_s^T\mathbf{F}(r,X_r,P_r,Q_r)dr-\int_s^T Q_rdB_r,\quad s\in[t,T],
	\end{aligned}
	\right.
\end{equation}
where $\mathbf{B},\mathbf{F}:[0,T]\times (L^2\times L^2\times(L^2)^n)(\Omega,\f,\mathbb{P};\brn)\to L^2(\Omega,\f,\mathbb{P};\brn)$, $\mathbf{A}: [0,T]\times (L^2\times L^2\times(L^2)^n)(\Omega,\f,\mathbb{P};\brn)\to L^2\left(\Omega,\f,\mathbb{P};\br^{n\times n}\right)$ and $\mathbf{G}:L^2(\Omega,\f,\mathbb{P};\brn)\to L^2(\Omega,\f,\mathbb{P};\brn)$. The system of FBSDEs \eqref{FB:11} can be viewed as a lifted version of FBSDEs \eqref{intro_2} and \eqref{FB:4} via \eqref{def:BAFG} and \eqref{def:BAFG'} respectively, under assumptions (A)'s and (B)'s to be stated in the following sections. The well-posedness of the generic FBSDEs \eqref{FB:11} requires  the following condition: 

\begin{condition}\label{Condition_mono}
There exists a map $\beta: [0,T]\times \left(L^2\times L^2\times(L^2)^n\right) (\Omega,\f,\mathbb{P};\brn)\ni(s,X,P,Q)\mapsto\beta(s,X,P,Q)\in L^2(\Omega,\f,\mathbb{P};\brd)$ and constants $\Lambda_\beta>0$, $\Gamma_\beta\geq 0$ and $K_\beta>0$, such that for any $X,X',P,P'\in L^2(\Omega,\f,\mathbb{P};\brn)$ and $Q,Q'\in L^2\left(\Omega,\f,\mathbb{P};\br^{n\times n}\right)$,
\begin{enumerate}[(i)]
    \item {\bf ($\beta$-Monotonicity)}
    \begin{enumerate}[(a)]
        \item The maps $\mathbf{B}$, $\mathbf{A}$ and $\mathbf{F}$ satisfy
	    \begin{align}
		    &\e\bigg[\left(\mathbf{F}(s,X',P',Q')-\mathbf{F}(s,X,P,Q)\right)^\top  (X'-X)\notag\\
            &\quad +\left(\mathbf{B}(s,X',P',Q')-\mathbf{B}(s,X,P,Q)\right)^\top  (P'-P) \notag \\[1mm]
		    &\quad +\sum_{j=1}^n \left(\mathbf{A}^j(s,X',P',Q')-\mathbf{A}^j(s,X,P,Q)\right)^\top  \left({Q'}^{j}-Q^j\right)\bigg] \notag \\
		    \le\ & -\Lambda_\beta \e\left[  \left|\beta(s,X',P',Q')-\beta(s,X,P,Q)\right|^2\right]\notag\\
		    & +\Gamma_\beta\left(\|X'-X\|^2_2+\|P'-P\|^2_2+\|Q'-Q\|_2^2 \right); \label{monotonicity} 
	    \end{align}
        \item The map $\mathbf{G}$ satisfies
        \begin{align*}
    	    &\e\left[\left(\mathbf{G}(X')-\mathbf{G}(X)\right)^\top  (X'-X)\right]\geq 0.
        \end{align*} 
    \end{enumerate}
    \item {\bf ($\beta$-Lipschitz)} The following continuity conditions hold
	\begin{align}
		&\|\mathbf{B}(s,X',P',Q')-\mathbf{B}(s,X,P,Q)\|^2_2+ \|\mathbf{A}(s,X',P',Q')-\mathbf{A}(s,X,P,Q)\|^2_2  \notag\\
		\le\ &   K_\beta \left(\|X'-X\|^2_2+\e\left[| \beta(s,X',P',Q')-\beta(s,X,P,Q)|^2\right]\right);\label{condi_Lip_1}\\[3mm]
		&\|\mathbf{F}(s,X',P',Q')-\mathbf{F}(s,X,P,Q)\|^2+\|\mathbf{G}(X')-\mathbf{G}(X)\|^2_2\notag\\
		\le\ &   K_\beta \left(\|X'-X\|^2_2+\|P'-P\|^2_2+\|Q'-Q\|_2^2 +\e \left[| \beta(X',s;P',Q')-\beta(X,s;P,Q)|^2 \right]\right). \notag
	\end{align}
\end{enumerate}
\end{condition}

We now state the well-posedness of FBSDEs \eqref{FB:11} under Condition \ref{Condition_mono}. 

\begin{lemma}\label{lem:1}
    Under Condition \ref{Condition_mono}, when 
    \begin{align}\label{big_Lam}
        \Lambda_\beta> 4T(T+1)K_\beta^2 \exp\left(2T(T+1)K_\beta^2 \right)\left[1+T(T+1)K_\beta^2 \exp\left(2T(T+1)K_\beta^2 \right)\right]\Gamma_\beta,
    \end{align}
    there is a unique adapted solution $(X,P,Q)\in\sr^2_\f(t,T)\times\sr^2_{\f}(t,T)\times\left(\lr^2_{\f}(t,T)\right)^n$ of the FBSDEs \eqref{FB:11}. Furthermore, if the parameter $\Gamma_\beta=0$ in Condition \ref{Condition_mono} (i)(a), then for any $\Lambda_\beta> 0$, the same assertion is still true.    
\end{lemma}

\begin{remark}
    The exponential term in \eqref{big_Lam} comes from the use of the Gr\"onwall's inequality. Actually, as a particular case, when the right hand side of \eqref{monotonicity} is simply:
    \begin{align}\label{monotonicity'}
         -\Lambda_\beta \e\left[\left|\beta(s,X',P',Q')-\beta(s,X,P,Q)\right|^2\right]+\Gamma_\beta \|X'-X\|^2_2,
    \end{align}
    and the right hand side of \eqref{condi_Lip_1} is just written as
    \begin{align}\label{condi_Lip_1'}
        K_\beta \ \e \left[| \beta(s,X',P',Q')-\beta(s,X,P,Q)|^2\right],
    \end{align}
    then the condition~\eqref{big_Lam} can much reduce to $\Lambda_{\beta} > 2T(T+1)K_\beta^2 \Gamma_\beta$; for instance, for the FBSDEs \eqref{intro_2} for the MFG \eqref{intro_MFG} in connection with the setting \eqref{def:BAFG}, if both the coefficient functions $b$ and $\sigma$ do not functionally depend on $x$ and $m$, then the the right hand side of \eqref{monotonicity} becomes  \eqref{monotonicity'} and the right hand side of \eqref{condi_Lip_1} reduces to \eqref{condi_Lip_1'}.
\end{remark}

The proof of the preceding well-posedness result appeals to  the method of continuation in coefficients proposed in \cite{YH2}, and is similar to that of \cite[Theorem 2.3]{SP}, \cite[Theorem 1]{SA1}, \cite[Lemma 4.1]{AB10} and \cite[Lemma 2.2]{AB11}, and is thus omitted here. Condition \ref{Condition_mono}(i) is actually the `usual monotonicity condition' for FBSDEs, which can be dated back to Hu-Peng \cite{YH2} and Peng-Wu \cite{SP} for fully coupled FBSDEs in Euclidean spaces. It is also referred as the weak monotonicity condition (displacement monotonicity condition) in Ahuja et al. \cite{SA1} for FBSDEs in Hilbert spaces, which can be applied to the FBSDEs arising from MFGs with a common noise. Here, our $\beta$-monotonicity in Condition \ref{Condition_mono} allows us to extend our study to more general situations. It is also used in our previous works \cite{AB10,AB11} to establish the well-posedness of the respective Jacobian and Hessian flows for FBSDEs \eqref{intro_2} and \eqref{FB:4} by choosing suitable maps $\beta$. In this paper, we further extend our previous works \cite{AB10,AB11} by relaxing assumptions on the running cost functional $f$ (see Section~\ref{sec:MFG}), and we can also include the generic drift functional cases (see Section~\ref{sec:generic}). Lemma~\ref{lem:1} with our $\beta$-monotonicity turns out to apply to all these cases.

\section{Monotonicity Conditions for MFG}\label{subsec:mono_mfg}\label{sec:MFG}

We now study the solvability of MFG \eqref{intro_MFG}. We first give a sufficient condition of maximum principle for MFG \eqref{intro_MFG}, and then introduce different monotonicity conditions for FBSDEs \eqref{intro_2}, and will give in the last section with the help of Lemma~\ref{lem:1} the well-posedness result under these monotonicity conditions. 

\subsection{Maximum principle}

We need the following assumptions on coefficients.

\textbf{(A1)} The functions $b$ and $\sigma$ are linear in $(x,v)$. That is,
\begin{align*}
	&b(s,x,m,v)=b_0(s,m)+b_1(s)x+b_2(s)v,\quad \sigma(s,x,m,v)=\sigma_0(s,m)+\sigma_1(s)x+\sigma_2(s)v.
\end{align*}
Here, the functions $b_0$ and $\sigma_0$ are $L$-Lipschitz continuous in $m\in\pr_2(\brn)$, and the norms of matrices $b_1(s),\sigma_1(s),b_2(s),\sigma_2(s)$ are bounded  by $L$. 

\textbf{(A2)} The cost functions $f$ and $g$ have a  quadratic growth, and  satisfy for $(s,x,m,v)\in [0,T]\times\brn\times \pr_2(\brn)\times \brd$,
\begin{align}
	|f(s,x,m,v)|\le L \left(1+|x|^2+W_2^2(m,\delta_0)+|v|^2\right),\quad |g(x,m)|\le L \left(1+|x|^2+W_2^2(m,\delta_0)\right). \label{quadratic_growth}
\end{align}
The derivatives $D_x f, \ D_v f, \ D_x g$ exist, and they are continuous in all their arguments, such that
\begin{align*}
    \left|(D_xf,D_vf) (s,x',m',v')-(D_xf,D_vf) (s,x,m,v)\right|\le\ & L\left(|x'-x|+|v'-v|+ W_2(m,m')\right),\\
    \left|D_x g(x',m')-D_x g(x,m)\right|\le\ & L\left(|x'-x|+ W_2(m,m')\right).
\end{align*}

\textbf{(A3)} (i) The terminal cost function $g$ is convex in $x$, and there exists $\lambda> 0$ such that for any $s\in[0,T]$ and  $(x,m,v,v')\in\brn\times\pr_2(\brn)\times\brd\times\brd$,
\begin{align*}
	&f(s,x,m,v')-f(s,x,m,v)\geq \left(D_v f (s,x,m,v)\right)^\top  (v'-v)+\lambda |v'-v|^2;
\end{align*}
(ii) moreover, the running cost function $f$ is jointly convex in $x$ and $v$, that is,
\begin{align*}
	f\left(s,x',v',m\right)-f(s,x,v,m)\geq \left[\begin{pmatrix}  D_x f\\ D_v f\end {pmatrix}(s,x,v,m)\right]^\top  \begin{pmatrix}  x'-x\\ v'-v\end {pmatrix}+\lambda \left|v'-v\right|^2.
\end{align*}

Under above assumptions, we define the optimal control $\hv$ as a map $[0,T]\times \brn\times\pr_2(\brn)\times \brn\times\br^{n\times n}\ni(s,x,m,p,q)\mapsto \widehat{v}(s,x,m,p,q)\in\brd$ so that
\begin{align}\label{hv}
	D_v L \left(s,x,m,\widehat{v}(s,x,m,p,q),p,q\right)=0.
\end{align}
We now give the sufficiency for the maximum principle for MFG \eqref{intro_MFG} under above assumptions.

\begin{theorem}\label{lem:MP1}
	Under Assumptions (A1), (A2) and (A3)(i), suppose that FBSDEs \eqref{intro_2} have a solution $(X,P,Q)\in \sr^2_\f(t,T)\times\sr^2_{\f}(t,T)\times\left(\lr^2_{\f}(t,T)\right)^n$. Then, there is a constant $c(L,T)>0$, such that for $\lambda>c(L,T)$, MFG \eqref{intro_MFG} has a unique solution
	\begin{align}\label{MFG_solution}
		v_s:=\hv\left(s,X_s,\lr(X_s),P_s,Q_s\right),\quad s\in[t,T].
	\end{align}
    Furthermore, if (A3)(ii) is satisfied, then, for any $\lambda> 0$, the same assertion still holds.
\end{theorem}

In Theorem~\ref{lem:MP1}, $c(L,T)=4L^2T(T+1) \exp (4L^2T(T+1))$, again the exponential term comes from the use of the Gr\"onwall's inequality. As a particular case, when both $b$ and $\sigma$ do not functionally depend on $x$, then one can take $c(L,T)=2L^2T(T+1)$. The proof of Theorem~\ref{lem:MP1} is similar to that of \cite[Lemma 2.1]{AB11}, and it is omitted here.

\subsection{Well-posedness of FBSDEs \eqref{intro_2}}

We shall use Lemma~\ref{lem:1} to give the well-posedness of our FBSDEs \eqref{intro_2}, which can be viewed as a subcase of FBSDEs \eqref{FB:11} by setting
\begin{equation}\label{def:BAFG}
	\begin{split}
		&\mathbf{B}(s,X,P,Q)(\omega):=D_p H(s,X(\omega),\lr(X),P(\omega),Q(\omega)),\\
		&\mathbf{A}(s,X,P,Q)(\omega):=D_q H(s,X(\omega),\lr(X),P(\omega),Q(\omega)),\\
		&\mathbf{F}(s,X,P,Q)(\omega):=-D_x H(s,X(\omega),\lr(X),P(\omega),Q(\omega)),\\
		&\mathbf{G}(X)(\omega):=D_x g(X(\omega),\lr(X)),\quad X,P\in L^2\left(\Omega,\f,\mathbb{P};\brn\right),\quad Q\in L^2\left(\Omega,\f,\mathbb{P};\br^{n\times n}\right).
	\end{split}
\end{equation}
From Assumptions (A1) and (A2), it is easy to check that the $\beta$-Lipschitz-continuities specified in Condition Condition~\ref{Condition_mono} (ii) are satisfied with the constant $ K_\beta=C(L)$ and the following choice of $\beta$:
\begin{align}\label{choice_of_beta}
    \beta(s,X,P,Q)(\omega):=\hv(s,X(\omega),\lr(X),P(\omega),Q(\omega)),
\end{align}
where the optimal control $\hv$ is defined in \eqref{hv}. In view of the map $G$ defined in \eqref{def:BAFG}, Condition~\ref{Condition_mono} (i) motivates us to introduce the following displacement quasi-monotonicity condition.
 
\begin{definition}[\bf Displacement quasi-monotonicity condition]\label{def:displacement_quasi}
    For a functional $g:\brn\times\pr_2(\brn)\to\br$ satisfying Assumption (A2), we say that $g$ satisfies the displacement quasi-monotonicity condition with a parameter $\lambda\in\br$, if 
    for any square-integrable random variables $\xi$ and $\xi'$ on the same probability space,
    \begin{align}
        \e\left[\left(D_x g \left(\xi',\lr(\xi')\right)-D_x g (\xi,\lr(\xi))\right)^\top  \left(\xi'-\xi\right)\right]\geq -\lambda \left\|\xi'-\xi\right\|_2^2. \label{g_mono-quasi}
    \end{align}
\end{definition}

\begin{remark}\label{def:displacement}
    Here, $\lambda$ can be an arbitrary real number. As one particular representative example of the Definition~\ref{def:displacement_quasi}, when $\lambda=0$ in \eqref{g_mono-quasi}, it is reduced to the prevalent displacement monotonicity condition:
    \begin{align}
        \e\left[\left(D_x g \left(\xi',\lr(\xi')\right)-D_x g (\xi,\lr(\xi))\right)^\top  \left(\xi'-\xi\right)\right]\geq 0. \label{g_mono}
    \end{align}
    This displacement monotonicity condition was first introduced in \cite{SA} under a different name ``weak monotonicity condition", and then it was also used in \cite{SA1,HZ1} for a  probabilistic approach to the solvability of MFG with a common noise. It is also used in \cite{MR4509653,GW} for an analytical method for the well-posedness of MFG master equations. 
\end{remark}

In view of the map $G$ defined in \eqref{def:BAFG}, we know that when $g$ satisfies the displacement monotonicity condition \eqref{g_mono}, we have
\begin{align}
	&\e\left[\left(\mathbf{G}(X')-\mathbf{G}(X)\right)^\top  (X'-X)\right] \notag \\
	=\ & \e\left[\left(D_xg\left(X',\lr(X')\right)-D_xg\left(X,\lr(X)\right)\right)^\top   \left(X'-X\right)\right] \geq 0, \label{use_latter}
\end{align} 
and therefore, Condition \ref{Condition_mono} (i)(b) is satisfied. In the rest of this section, we focus on conditions on $f$ to ensure the $\beta$-monotonicity in Condition~\ref{Condition_mono} (i)(a) with the choice of $\beta$ in \eqref{choice_of_beta}. Then, the well-posedness of FBSDEs \eqref{intro_2} can be deduced by Lemma~\ref{lem:1}. The following result shows that the convexity of $f$ in $v$ in Assumption (A3)(i) can guarantee Condition~\ref{Condition_mono} (i)(a), and then gives the local solvability of FBSDEs \eqref{intro_2}. 

\begin{theorem}\label{local_solva_mfg}
    Under Assumptions (A1), (A2) and (A3)(i), coefficients in \eqref{def:BAFG} satisfy Condition \ref{Condition_mono} with the choice of $\beta$ in \eqref{choice_of_beta}. As a consequence, there is a constant $c(L,T)>0$, such that for $\lambda>c(L,T)$, there is a unique adapted solution of the FBSDEs \eqref{intro_2}. 
\end{theorem}

\begin{proof}
    From a similar approach as \cite[Lemma 2.3]{AB11}, we can deduce that Condition \ref{Condition_mono}(i)(a) is satisfied with $\Lambda_\beta=\lambda$, $\Gamma_\beta=\frac{9L^2}{4\lambda}+4L$, and $K_\beta=4L$. Therefore, from \eqref{big_Lam}, we obtain the desired result. Here, $c(L,T)=A+\sqrt{A^2+4A}$, where $A:= 64T(T+1)L^3 \exp\left(32T(T+1)L^2 \right)[1+16T(T+1)L^2 \exp\left(32T(T+1)L^2 \right)]$. The exponential term comes from the use of the Gr\"onwall's inequality. As a particular case, when both $b$ and $\sigma$ do not functionally depend on $x$ and $m$, we can take $A=2T(T+1)L^3$. This is similar to the condition (3.11) in \cite{AB9'}.    
\end{proof}

Theorem~\ref{local_solva_mfg} only gives the solvability of FBSDEs \eqref{intro_2} when $T$ is small enough. From now on, we aim to obtain the global well-posedness of FBSDEs \eqref{intro_2}, which requires Condition~\ref{Condition_mono} (i)(a) to be satisfied with the parameter $\Gamma_\beta=0$ for our settings in \eqref{def:BAFG} and \eqref{choice_of_beta}. From Assumption (A1), when the coefficients $b_0$ and $\sigma_0$ do not functionally depend on $m$, for $s\in[t,T]$, $X,X',P,P'\in L^2\left(\Omega,\f,\mathbb{P};\brn\right)$ and $Q,Q'\in \left(L^2\left(\Omega,\f,\mathbb{P};\brn\right)\right)^n$, by denoting $\hv:=\hv(s,X,\lr(X),P,Q)$ and $\hv':=\hv\left(s,X',\lr(X'),P',Q'\right)$, we can compute that
\begin{align}
	&\e\bigg[\left(\mathbf{F}(s,X',P',Q')-\mathbf{F}(s,X,P,Q)\right)^\top  (X'-X)+\left(\mathbf{B}(s,X',P',Q')-\mathbf{B}(s,X,P,Q)\right)^\top  (P'-P) \notag \\
	&\quad +\sum_{j=1}^n \left(\mathbf{A}^j(s,X',P',Q')-\mathbf{A}^j(s,X,P,Q)\right)^\top  \left({Q'}^{j}-Q^j\right)\bigg] \notag \\
    =\ & -\e\left\{ \left[\begin{pmatrix}  D_x f\\ D_v f\end {pmatrix}(s,\cdot)\bigg|^{\left(X',\lr(X'),\hv\left(s,X',\lr(X'),P',Q'\right)\right)}_{\left(X,\lr(X),\hv(s,X,\lr(X),P,Q)\right)} \right]^\top  \begin{pmatrix}  X'-X\\ \hv'-\hv\end {pmatrix}\right\}, \label{compute_1}
\end{align}    
with which we shall introduce different monotonicity assumptions on $f$ so as to ensure the right side of \eqref{compute_1} to be smaller than $\e\left[-\Lambda_\beta\left|\hv\left(s,X',\lr(X'),P',Q'\right)-\hv(s,X,\lr(X),P,Q))\right|^2\right]$ for some $\Lambda_\beta>0$. Based on the displacement quasi-monotonicity condition in Definition~\ref{def:displacement_quasi}, we first propose the following weak monotonicity condition on $f$ for separable case first, and then in Condition~\ref{assumption_mono_1} we shall introduce the one for non-separable case.

\begin{condition}[\bf Displacement quasi-monotonicity for separable cases]\label{condi:SA1}
    The functional $f$ is separable in $m$ and $v$: $f(s,x,m,v)=f_0(s,x,m)+f_1(s,x,v)$, where the functional $f_1$ satisfies the following strong convexity: there exists $\lambda_v>0$ and $\lambda_x\geq 0$, such that 
    \begin{align}
	    f_1\left(s,x',v'\right)-f_1(s,x,v)\geq\ & \left[\begin{pmatrix}  D_x f_1\\ D_v f_1\end {pmatrix}(s,x,v)\right]^\top  \begin{pmatrix}  x'-x\\ v'-v\end {pmatrix}+\lambda_x \left|x'-x\right|^2 +\lambda_v \left|v'-v\right|^2;\label{strong_convex_0}
    \end{align}
    the functional $f_0$ satisfies the following displacement quasi-monotonicity: there exists $\lambda_m\le 2\lambda_x$, such that
    \begin{equation}\label{dis_1}
    \begin{split}
        \e\left[\left(D_x f_0 \left(s, \xi', \lr(\xi')\right)-D_x f_0 (s,\xi,\lr(\xi))\right)^\top  \left(\xi'-\xi\right)\right]\geq -\lambda_m \left\|\xi'-\xi\right\|_2^2,\quad \forall \xi,\xi'\in L^2(\Omega,\f,\mathbb{P};\brn).
    \end{split}
    \end{equation}
\end{condition}

\begin{remark}
    Condition~\ref{condi:SA1} allows $\lambda_m$ to be positive, that is, the running cost functional $f_0$ can be displacement quasi-monotonic. Particularly, when $\lambda_m=0$, \eqref{dis_1} is reduced to the classical displacement monotonicity condition as in Remark~\ref{def:displacement}. In this case, we can allow $\lambda_x=0$, and then, Condition~\ref{condi:SA1} coincides with the assumption on $f$ in \cite{SA1}.
\end{remark}

We next show that Condition~\ref{condi:SA1} can ensure our $\beta$-monotonicity in Condition~\ref{Condition_mono} with $\Gamma_\beta=0$.

\begin{theorem}\label{thm:condi:sepa}
    Under Assumptions (A1)-(A3), suppose that the functions $b_0$ and $\sigma_0$ do not functionally depend on $m$, the terminal cost functional $g$ satisfies the displacement monotonicity condition \eqref{g_mono}, and the running cost functional $f$ satisfies Condition~\ref{condi:SA1}. Then, coefficients in \eqref{def:BAFG} satisfy Condition~\ref{Condition_mono} (i)(a) with with the choice of $\beta$ in \eqref{choice_of_beta} and the parameter $\Gamma_\beta=0$. As a consequence, FBSDEs \eqref{intro_2} have a unique global solution.
\end{theorem}

\begin{proof}
For the sake of notational convenience, in this proof, we denote by  $\hv:=\hv(s,X,\lr(X),P,Q)$ and $\hv':=\hv\left(s,X',\lr(X'),P',Q'\right)$. From \eqref{compute_1} and the separable assumption in Condition~\ref{condi:SA1}, we have
\begin{align*}
	&\e\bigg[\left(\mathbf{F}(s,X',P',Q')-\mathbf{F}(s,X,P,Q)\right)^\top  (X'-X)+\left(\mathbf{B}(s,X',P',Q')-\mathbf{B}(s,X,P,Q)\right)^\top  (P'-P) \notag \\
	&\quad +\sum_{j=1}^n \left(\mathbf{A}^j(s,X',P',Q')-\mathbf{A}^j(s,X,P,Q)\right)^\top  \left({Q'}^{j}-Q^j\right)\bigg] \notag \\
    \ =& -\e\left\{ \left[\begin{pmatrix}  D_x f_1\\ D_v f_1\end {pmatrix}(s,\cdot)\bigg|^{\left(X',\hv'\right)}_{\left(X,\hv\right)} \right]^\top  \begin{pmatrix}  X'-X\\ \hv'-\hv\end{pmatrix} +\left(D_xf_0(s,\cdot)\Big|^{(X',\lr(X'))}_{(X,\lr(X))} \right)^\top  (X'-X)\right\}.
\end{align*}    
From the displacement quasi-monotonicity \eqref{dis_1}, we know that
\begin{align*}
    -\e\Big[\left(D_xf_0(s,X',\lr(X'))-D_xf_0(s,X,\lr(X))\right)^\top  (X'-X)\Big]\le \lambda_m \left\|\xi'-\xi\right\|_2^2,
\end{align*}
and from the convexity \eqref{strong_convex_0}, we have
\begin{align*}
    &-\e\Big[\left(D_vf_1\left(s,X',\hv'\right)-D_vf_1\left(s,X,\hv \right)\right)^\top   \left(\hv'-\hv\right) +\left(D_xf_1(s,X',\hv')-D_xf_1(s,X,\hv)\right)^\top  (X'-X)\Big]\\
    \le\ & -2\lambda_x\|\xi'-\xi\|_2^2-2\lambda_v\left\|\hv'-\hv\right\|_2^2.
\end{align*}
Combining the last two inequalities, we can see that
\begin{align*}
    & -\e\left\{ \left[\begin{pmatrix}  D_x f_1\\ D_v f_1\end {pmatrix}(s,\cdot)\bigg|^{\left(X',\hv'\right)}_{\left(X,\hv\right)} \right]^\top  \begin{pmatrix}  X'-X\\ \hv'-\hv\end{pmatrix} +\left(D_xf_0(s,\cdot)\Big|^{(X',\lr(X'))}_{(X,\lr(X))} \right)^\top  (X'-X)\right\}\\
    \le\ & (\lambda_m-2\lambda_x)\|\xi'-\xi\|_2^2-2\lambda_v\left\|\hv'-\hv\right\|_2^2,
\end{align*}
and since $\lambda_m\le 2\lambda_x$, we know that Condition \ref{Condition_mono} (i)(a) is valid with the choice of $\beta$ in \eqref{choice_of_beta} for $\Lambda_\beta=2\lambda_v$ and $\Gamma_\beta=0$. 
\end{proof}

We now consider the case when $f$ is not separable in $x$ and $v$. We propose the following strong convexity and small mean field effect condition, which is also used in our previous work \cite{AB11}.

\begin{condition}[\bf Strong convexity and small mean field effect]\label{def:small_mf}
    There exist nonnegative constants $L_x,L_v\le L$, such that 
    \begin{equation}\label{small_mean_field}
    \begin{split}
        \left|D_x f(s,x',m',v')-D_x f(s,x,m,v)\right|\le\ & L\left(|x'-x|+|v'-v|\right)+L_x W_2(m,m'),\\
        \left|D_v f(s,x',m',v')-D_v f(s,x,m,v)\right|\le\ & L\left(|x'-x|+|v'-v|\right)+L_v W_2(m,m');
    \end{split}
    \end{equation}
    and there exist $\lambda_v>0$ and $\lambda_x\geq\frac{L_v^2}{8\lambda_v}+\frac{L_x}{2}$ such that
    \begin{align}
	    &f\left(s,x',m,v'\right)-f(s,x,m,v)\geq \left[ \begin{pmatrix} D_x f\\ D_v f \end{pmatrix}(s,x,m,v)\right]^\top \begin{pmatrix} x'-x\\ v'-v \end{pmatrix}+\lambda_x \left|x'-x\right|^2 +\lambda_v \left|v'-v\right|^2.\label{strong_convex_1}
    \end{align}
\end{condition}

Here, in the relation $\lambda_x\geq\frac{L_v^2}{8\lambda_v}+\frac{L_x}{2}$, the parameter $\frac{1}{8}$ is not the optimal, but we do not drill down into the details. We first give the relation of Condition~\ref{def:small_mf} with the displacement monotonicity condition when $f$ is independent of $v$.

\begin{remark}
    For a functional $f$ satisfying Condition~\ref{def:small_mf} and independent of $v$, the continuity \eqref{small_mean_field} reduces to
    \begin{align}\label{rk:1_1}
        \left|D_x f(s,x',m')-D_x f(s,x,m)\right|\le L|x'-x|+L_x W_2(m,m'),
    \end{align}
    and the convexity \eqref{strong_convex_1} reduces to
    \begin{align}\label{rk:1_2}
        &f\left(s,x',m\right)-f(s,x,m)\geq \left(D_x f (s,x,m)\right)^\top  \left(x'-x\right)+\lambda_x \left|x'-x\right|^2,
    \end{align}
    with the constants satisfying $\lambda_x\geq \frac{L_x}{2}$. For any square-integrable random variables $\xi$ and $\xi'$ on the same probability space, we know that
    \begin{align*}
        &\e\left[\left(D_x f \left(s,\xi',\lr(\xi')\right)-D_x f (s,\xi,\lr(\xi))\right)^\top  \left(\xi'-\xi\right)\right]\\
        =\ &\e\left[\left(D_x f \left(s,\xi',\lr(\xi')\right)-D_x f (s,\xi,\lr(\xi'))\right)^\top  \left(\xi'-\xi\right)\right]\\
        &+\e\left[\left(D_x f \left(s,\xi,\lr(\xi')\right)-D_x f (s,\xi,\lr(\xi))\right)^\top  \left(\xi'-\xi\right)\right].
    \end{align*}
    From the convexity \eqref{rk:1_2}, we know that
    \begin{align*}
        &\e\left[\left(D_x f \left(s,\xi',\lr(\xi')\right)-D_x f (s,\xi,\lr(\xi'))\right)^\top  \left(\xi'-\xi\right)\right]\geq  2\lambda_x \left\|\xi'-\xi\right\|_2^2,
    \end{align*}
    and from \eqref{rk:1_1}, we can compute that
    \begin{align*}
        \e\left[\left(D_x f \left(s,\xi,\lr(\xi')\right)-D_x f (s,\xi,\lr(\xi))\right)^\top  \left(\xi'-\xi\right)\right]\geq -L_x \left\|\xi'-\xi\right\|_2^2.
    \end{align*}
    Combining the last two inequalities, we have 
    \begin{align*}
        &\e\left[\left(D_x f \left(s,\xi',\lr(\xi')\right)-D_x f (s,\xi,\lr(\xi))\right)^\top  \left(\xi'-\xi\right)\right]\geq (2\lambda_x-L_x) \left\|\xi'-\xi\right\|_2^2;
    \end{align*}
    since $\lambda_x\geq \frac{L_x}{2}$, we know that $f$ satisfies the displacement monotonicity condition. 
    
    However, when $f$ is dependent on $v$, Condition~\ref{def:small_mf} cannot be included by Condition~\ref{condi:SA1}. For the very special linear-quadratic case $f(x,m,v):=|x|^2+|v|^2+v^\top \int_{\brn}y m(dy)$, it can be easily to check that $f$ satisfies Condition~\ref{def:small_mf}, yet $f$ is not separable in $m$ and $v$.
\end{remark}

We now show that Condition~\ref{def:small_mf} can also ensure our $\beta$-monotonicity in Condition~\ref{Condition_mono} with $\Gamma_\beta=0$.

\begin{theorem}\label{thm:FBSDES_small_mf}
    Under Assumptions (A1)-(A3), suppose that the functions $b_0$ and $\sigma_0$ do not functionally depend on $m$, the terminal cost functional $g$ satisfies the displacement monotonicity condition \eqref{g_mono}, and the running cost functional $f$ satisfies Condition~\ref{def:small_mf}. Then, coefficients in \eqref{def:BAFG} satisfy Condition~\ref{Condition_mono} (i)(a) with with the choice of $\beta$ in \eqref{choice_of_beta} and the parameter $\Gamma_\beta=0$. As a consequence, FBSDEs \eqref{intro_2} have a unique global solution.
\end{theorem}

\begin{proof}
As in the proof of Theorem~\ref{thm:condi:sepa}, we adopt the use the notations $\hv:=\hv(s,X,\lr(X),P,Q)$ and $\hv':=\hv\left(s,X',\lr(X'),P',Q'\right)$. Since $f$ satisfies \eqref{small_mean_field} and \eqref{strong_convex_1}, we can deduce that
\begin{align}
		& -\e\left\{\left[\begin{pmatrix}  D_x f\\ D_v f\end {pmatrix}(s,\cdot)\bigg|^{\left(X',\lr(X'),\hv'\right)}_{\left(X,\lr(X),\hv\right)} \right]^\top  \begin{pmatrix}  X'-X\\ \hv'-\hv\end{pmatrix} \right\} \notag \\
		\le \ &-\e\left\{\left[\begin{pmatrix}  D_x f\\ D_v f\end {pmatrix}(s,\cdot,\lr(X'),\cdot)\bigg|^{\left(X',\hv'\right)}_{\left(X,\hv\right)} \right]^\top  \begin{pmatrix}  X'-X\\ \hv'-\hv\end{pmatrix} \right\} \notag \\
        &+\e\left\{\left|\left[\begin{pmatrix}  D_x f\\ D_v f\end {pmatrix}(s,X,\cdot,\hv')\bigg|^{\lr(X')}_{\lr(X)} \right]^\top  \begin{pmatrix}  X'-X\\ \hv'-\hv\end{pmatrix} \right|\right\} \notag \\
		\le\ & -2 \e\left[\lambda_v \left| \hv'-\hv\right|^2+\lambda_x \left| X'-X \right|^2\right] +L_v \|X'-X\|_2 \cdot \left\|\hv'-\hv\right\|_2 +L_x \|X'-X\|_2^2. \label{small_mf_eff_3}
\end{align} 
Then, we see from \eqref{compute_1} and Young's inequality that
\begin{align*}
    &\e\bigg[\left(\mathbf{F}(s,X',P',Q')-\mathbf{F}(s,X,P,Q)\right)^\top  (X'-X)+\left(\mathbf{B}(s,X',P',Q')-\mathbf{B}(s,X,P,Q)\right)^\top  (P'-P) \notag \\
	&\quad +\sum_{j=1}^n \left(\mathbf{A}^j(s,X',P',Q')-\mathbf{A}^j(s,X,P,Q)\right)^\top  \left({Q'}^{j}-Q^j\right)\bigg] \notag \\
	\le\ & -2 \e\left[\lambda_v \left| \hv'-\hv\right|^2+\lambda_x \left| X'-X \right|^2\right] +L_v \|X'-X\|_2 \cdot \left\|\hv'-\hv\right\|_2 +L_x \left\|X'-X\right\|_2^2\\
    \le\ & -\lambda_v  \left\| \hv'-\hv\right\|^2_2-\left(2\lambda_x-L_x-\frac{L_v^2}{4\lambda_v}\right)\left\| X'-X \right\|^2_2 \\
    \le\ & -\lambda_v \left\| \hv'-\hv\right\|_2^2.
\end{align*}
In the last inequality, we have used the inequality in Condition~\ref{def:small_mf}: $\lambda_x\geq \frac{L_x}{2}+\frac{L_v^2}{8\lambda_v}$. Therefore, we know that Condition \ref{Condition_mono} (i)(a) is valid with the choice of $\beta$ in \eqref{choice_of_beta} by setting $\Lambda_\beta=\lambda_v$ and $\Gamma_\beta=0$. 
\end{proof}

\begin{remark}
    Although we here assume that $b_0$ and $\sigma_0$ do not functionally depend on $m$, since the functional $f$ here can be nonseparable, therefore, the optimal control $\hv$ here can depend on $m$, and hence, the coefficients
    \begin{align*}
        D_p H(s,x,m,p,q)=b\left(s,x,m,\hv(s,x,m,p,q)\right),\quad D_q H(s,x,m,p,q)=\sigma\left(s,x,m,\hv(s,x,m,p,q)\right)
    \end{align*}
    are dependent on $m$. Therefore, the coefficients of the SDE in the system \eqref{intro_2} is distribution-dependent. Moreover, since we allow the diffusion term $\sigma$ to depend on $v$, the optimal control $\hv$ can depend on $q$, which also extends the results in the existing literature via a PDE approach.
\end{remark}

In Condition~\ref{def:small_mf}, the functional $f$ does not need to be separable or displacement monotonic. From the relation $\lambda_x\geq\frac{L_v^2}{8\lambda_v}+\frac{L_x}{2}$, we can see that the dependence of $D_xf$ and $D_vf$ on $m$ should be smaller than the convexity of $f$ in $x$. Now, we further extend Conditions~\ref{condi:SA1} and \ref{def:small_mf} and propose the following monotonicity assumption for $f$. The functional $f$ is assumed to be divided into two parts (both dependent on $(s,x,m,v)$), with one part satisfying a strong convexity and small mean field effect condition, and the other part satisfying a displacement quasi-monotonicity. 

\begin{condition}\label{assumption_mono_1}
    The running cost functional $f$ is in the form:
    \begin{align}\label{condition_divide}
        f(s,x,m,v)=f_0(s,x,m,v)+f_1(s,x,m,v).
    \end{align}
    Here, the functional $f_1$ satisfies the continuity \eqref{small_mean_field} with nonnegative constants $L_x$ and $L_v$, and it also satisfies the convexity \eqref{strong_convex_1} with constants $\lambda_x\geq 0$ and $\lambda_v>0$. The functional $f_0$ is convex in $v$, and satisfies the following displacement quasi-monotonicity with a constant $\lambda_m\in\br$: for any $(s,V)\in[0,T]\times L^2(\Omega,\f,\mathbb{P};\brd)$, 
    \begin{equation}\label{random_displacement_0}
    \begin{split}
        \e\left[\left(D_x f_0 \left(s,\xi',\lr(\xi'),V\right)-D_x f_0 (s,\xi,\lr(\xi),V)\right)^\top  \left(\xi'-\xi\right)\right]\geq -\lambda_m \left\|\xi'-\xi\right\|_2^2 ,\\ 
        \forall \xi,\xi'\in L^2(\Omega,\f,\mathbb{P};\brn);
    \end{split}
    \end{equation}
    and there exists nonnegative constant $l_x$, such that 
    \begin{equation}\label{large_mono_lip_1}
    \begin{split}
        \left|D_x f_0(s,x,m,v')-D_x f_0(s,x,m,v)\right|\le\ & l_x|v'-v|,\\
        \left|D_v f_0(s,x',m',v)-D_v f_0(s,x,m,v)\right|\le\ & l_x\left(|x'-x|+W_2(m,m')\right).
    \end{split}
    \end{equation}   
    These parameters satisfy the following inequality condition:
    \begin{align}\label{parameters_condition}
        2\lambda_x-\lambda_m\geq L_x+\frac{(L_v+3l_x)^2}{4\lambda_v}.
    \end{align}    
\end{condition}

In Condition~\ref{assumption_mono_1}, we see that the dependence of $D_xf_0$ in $m$ can be large, and the $f_1$ part does not need to be monotonic. From \eqref{parameters_condition}, we can see that the strong convexity of $f_1$ in $x$ allows $f_0$ to be quasi-monotonic in $m$. Condition~\ref{assumption_mono_1} includes Condition~\ref{condi:SA1} as a special case. Actually, when $f_0$ is independent of $v$, then the parameters $l_x=0$, and the monotonicity condition \eqref{random_displacement_0} for $f_0$ is reduced to \eqref{dis_1}. Furthermore, if $f_1$ does not functionally depend on $m$, then the parameters $L_x=L_v=0$, and the convexity condition \eqref{strong_convex_1} for $f_1$ is reduced to \eqref{strong_convex_0}. Therefore, in this case, Condition~\ref{assumption_mono_1} is reduced to Condition~\ref{condi:SA1}. We also elaborate more on the condition \eqref{random_displacement_0}. When the derivative $D_x f_0$ is differentiable in $x$ and is linearly functionally differentiable in $m$ (although we do not need such regularity in this article), in view of \eqref{lem01_1}, the inequality \eqref{random_displacement_0} is equivalent to the following one: for any $(s,V)\in[0,T]\times L^2(\Omega,\f,\mathbb{P};\brd)$ and $\xi,\eta\in L^2(\Omega,\f,\mathbb{P};\brn)$,
\begin{align}\label{random_displacement_1}
    &\e\left\{\eta^\top \left[D_{x}^2 f_0 \left(s,\xi,\lr(\xi),V\right)\eta +\widetilde{\e}\left(D_xD_y\frac{d f_0}{d\nu}(s,\xi,\lr(\xi),V)\big(\widetilde{\xi}\big)\widetilde{\eta}\right)+\lambda_m\eta \right]\right\}\geq 0.
\end{align}
In particular, when the derivatives $D_x^2 f_0$ and $D_xD_y\frac{d f_0}{d\nu}$ do not functionally depend on $v$, then \eqref{random_displacement_1} is reduced to the displacement quasi-monotonicity condition as in \eqref{dis_1}. We also refer to \cite[Definition 3.4]{GW} for a similar assumption as \eqref{random_displacement_1} on the Hamiltonian functional $H$.

We next show that Condition~\ref{assumption_mono_1} yields our $\beta$-monotonicity in Condition~\ref{Condition_mono} with $\Gamma_\beta=0$, and then FBSDEs \eqref{intro_2} is well-posed.

\begin{theorem}\label{thm:mono}
    Under Assumptions (A1)-(A3), suppose that the functions $b_0$ and $\sigma_0$ do not functionally depend on $m$, the terminal cost functional $g$ satisfies the displacement monotonicity condition \eqref{g_mono}, and the running cost functional $f$ satisfies Condition~\ref{assumption_mono_1}. Then, coefficients in \eqref{def:BAFG} satisfy Condition~\ref{Condition_mono} (i)(a) with with the choice of $\beta$ in \eqref{choice_of_beta} and the parameter $\Gamma_\beta=0$. As a consequence, FBSDEs \eqref{intro_2} have a unique global solution.
\end{theorem}

\begin{proof}
We still use the notations $\hv:=\hv(s,X,\lr(X),P,Q)$ and $\hv':=\hv\left(s,X',\lr(X'),P',Q'\right)$. From \eqref{compute_1} and \eqref{condition_divide}, we know that
\begin{align}
	&\e\bigg[\left(\mathbf{F}(s,X',P',Q')-\mathbf{F}(s,X,P,Q)\right)^\top  (X'-X)+\left(\mathbf{B}(s,X',P',Q')-\mathbf{B}(s,X,P,Q)\right)^\top  (P'-P) \notag \\
	&\quad +\sum_{j=1}^n \left(\mathbf{A}^j(s,X',P',Q')-\mathbf{A}^j(s,X,P,Q)\right)^\top  \left({Q'}^{j}-Q^j\right)\bigg] \notag \\
    \ =& -\e\Bigg\{\Bigg[\begin{pmatrix}  D_x f_1\\ D_v f_1\end {pmatrix}(s,\cdot)\Bigg|^{\left(X',\lr(X'),\hv'\right)}_{\left(X,\lr(X),\hv\right)} \Bigg]^\top  \begin{pmatrix}  X'-X\\ \hv'-\hv\end{pmatrix} +\Bigg[\begin{pmatrix}  D_x f_0\\ D_v f_0\end {pmatrix}(s,\cdot)\Bigg|^{\left(X',\lr(X'),\hv'\right)}_{\left(X,\lr(X),\hv\right)} \Bigg]^\top  \begin{pmatrix}  X'-X\\ \hv'-\hv\end{pmatrix} \Bigg\}.\label{lem2_3}
\end{align} 
Since $f_1$ satisfies \eqref{small_mean_field} and \eqref{strong_convex_1}, from \eqref{small_mf_eff_3}, we then obtain
\begin{align}
		& -\e\Bigg\{\Bigg[\begin{pmatrix}  D_x f_1\\ D_v f_1\end {pmatrix}(s,\cdot)\Bigg|^{\left(X',\lr(X'),\hv'\right)}_{\left(X,\lr(X),\hv\right)} \Bigg]^\top  \begin{pmatrix}  X'-X\\ \hv'-\hv\end{pmatrix} \Bigg\}  \notag \\
		\le\ & -2 \e\left[\lambda_v \left| \hv'-\hv\right|^2+\lambda_x \left| X'-X \right|^2\right] +L_v \left\|X'-X\right\|_2 \left\|\hv'-\hv\right\|_2 +L_x \left\|X'-X\right\|_2^2. \label{small_mf_eff_3'}
\end{align} 
Since $f_0$ satisfies the displacement quasi-monotonicity condition \eqref{random_displacement_0}, we have
\begin{align}\label{thm:large_mono_1}
    -\e\Big[\left(D_xf_0\left(s,X',\lr(X'),\hv\right)-D_xf_0\left(s,X,\lr(X),\hv \right)\right)^\top   \left(X'-X\right) \Big]\le \lambda_m \left\|X'-X\right\|_2^2.
\end{align}
As $f_0$ is convex in $v$, we know that
\begin{align}
    &-\e\Big[\left(D_vf_0\left(s,X,\lr(X),\hv'\right)-D_vf_0\left(s,X,\lr(X),\hv \right)\right)^\top   \left(\hv'-\hv\right) \Big]\le 0. \label{thm:large_mono_2}
\end{align}
Combining \eqref{large_mono_lip_1} with an application of Cauchy-Schwarz inequality, we also have
\begin{equation}\label{thm:large_mono_4}
\begin{aligned}
    -\e\left\{\left( D_v f_0 (s,\cdot,\hv')\Big|^{\left(X',\lr(X')\right)}_{\left(X,\lr(X)\right)} \right)^\top  \left( \hv'-\hv\right) \right\} \le \ &  2l_x \left\|X'-X\right\|_2 \left\|\hv'-\hv\right\|_2, \\
    -\e\left\{\left( D_x f_0 \left(s,X',\lr(X'),\cdot\right)\left|^{\hv'}_{\hv} \right.\right)^\top  \left( X'-X\right) \right\} \le \ & l_x \left\|X'-X\right\|_2 \left\|\hv'-\hv\right\|_2. 
\end{aligned}
\end{equation}
From \eqref{thm:large_mono_1}-\eqref{thm:large_mono_4}, we deduce that
\begin{align}
    &-\e\Bigg\{\Bigg[\begin{pmatrix}  D_x f_0\\ D_v f_0\end {pmatrix}(s,\cdot)\bigg|^{\left(X',\lr(X'),\hv'\right)}_{\left(X,\lr(X),\hv\right)} \Bigg]^\top  \begin{pmatrix}  X'-X\\ \hv'-\hv\end{pmatrix} \Bigg\} \le \lambda_m \left\|X'-X\right\|_2^2+3l_x \left\|X'-X\right\|_2 \left\|\hv'-\hv\right\|_2. \label{small_mf_eff_8}
\end{align}
Combining \eqref{lem2_3}, \eqref{small_mf_eff_3'} and \eqref{small_mf_eff_8}, from the Young's inequality, we have
\begin{align}
    &\e\bigg[\left(\mathbf{F}(s,X',P',Q')-\mathbf{F}(s,X,P,Q)\right)^\top  (X'-X)+\left(\mathbf{B}(s,X',P',Q')-\mathbf{B}(s,X,P,Q)\right)^\top  (P'-P)  \notag \\
	&\quad +\sum_{j=1}^n \left(\mathbf{A}^j(s,X',P',Q')-\mathbf{A}^j(s,X,P,Q)\right)^\top  \left({Q'}^{j}-Q^j\right)\bigg]  \notag \\
    \le\ & -2\lambda_v \left\| \hv'-\hv\right\|_2^2 -(2\lambda_x-\lambda_m-L_x) \left\| X'-X \right\|_2^2 +(L_v+3l_x) \left\|X'-X\right\|_2 \left\|\hv'-\hv\right\|_2 \label{small_mf_eff_9} \\
    \le\ & -\lambda_v \left\|\hv'-\hv\right\|_2^2 -\left(2\lambda_x-\lambda_m-L_x-\frac{(L_v+3l_x)^2}{4\lambda_v}\right)\left\| X'-X \right\|_2^2 \notag\\
    \le\ &  -\lambda_v \left\|\hv'-\hv\right\|_2^2, \notag
\end{align}
where the last inequality follows in light of $2\lambda_x-\lambda_m\geq L_x+\frac{(L_v+3l_x)^2}{4\lambda_v}$. Therefore, Condition \ref{Condition_mono} (i)(a) is valid with the choice of $\beta$ in \eqref{choice_of_beta} for $\Lambda_\beta=\lambda_v$ and $\Gamma_\beta=0$. 
\end{proof}

In the proof of Theorem~\ref{thm:mono}, we have shown that Condition~\ref{assumption_mono_1} for the running cost functional $f$ ensure the $\beta$-monotonicity in Condition~\ref{Condition_mono} holds with the parameter $\Gamma_\beta=0$. As two particular cases of Condition~\ref{assumption_mono_1}, the classical displacement monotonicity \cite{SA,SA1,MR4509653,GW} and the small mean field effect \cite{AB11} can both be viewed as a condition to guarantee this $\beta$-monotonicity. We also refer to \cite[Remark 7.2]{AB10'} and \cite[Subsection 2.1]{AB9'} for more detailed discussions on the relationship between the small mean field effect and displacement monotonicity condition. As a direct consequence of Theorems~\ref{lem:MP1}, \ref{local_solva_mfg} and \ref{thm:mono}, we obtain the following solvability of MFG \eqref{intro_MFG}.

\begin{corollary}
    Under Assumptions (A1)-(A3), \eqref{MFG_solution} gives a local solution of MFG \eqref{intro_MFG}. Furthermore, under assumptions in Theorem~\ref{thm:mono}, \eqref{MFG_solution} is the unique global solution of MFG \eqref{intro_MFG}.
\end{corollary}

So far, we focus on the convexity setting on the terminal functional $g$, actually, our stochastic control method can be applied to some more general cases beyond usual convex settings. To this point, we first introduce a recent work of \cite{Anti_mono}, in which the authors use an analytical method to study the MFG master equation with $\sigma$ being a constant and in the absence of the displacement monotonicity condition; particularly, they assume certain non-convexity of the terminal cost $g$. More precisely, $g$ is assumed to satisfy an anti-monotonicity condition, while the Hamiltonian $H$ taken the form $H(s,m,p)=\langle A_0 x,p\rangle+H_0(x,m,p)$, where $H_0$ satisfies appropriate regularity conditions; see \cite[Assumption 6.1]{Anti_mono} for details. 

We now give a viewpoint of the above conditions from the control theoretical perspective. In our previous work \cite{AB11}, we can see that with a stochastic control approach, the convexity of $g$ is also not necessary for the local solvability of the FBSDEs associated with the MFG. Actually, even for the global solvability, the convexity and the displacement monotonicity condition of $g$ are not necessary either; indeed, we may use the following condition on $g$: there exists some $l_g>0$, such that for any $\xi$ and $\xi'$ on the same probability space,
\begin{align}
    \left|\e\left[\left(D_x g \left(\xi',\lr(\xi')\right)-D_x g (\xi,\lr(\xi))\right)^\top  \left(\xi'-\xi\right)\right]\right| \le l_g \left\|\xi'-\xi\right\|_2^2, \label{g_generic_condition}
\end{align}
which resembles to the anti-monotonicity condition \cite[(1.5)]{Anti_mono}. Here, we assume the running cost functional $f$ satisfies Condition~\ref{assumption_mono_1}, which allows the Hamiltonian $H$ to be a general one; and due to the page limit, we only give a sketch of motivation for a simple case with $\sigma$ being constant. In view of the continuation method, we should study the continuity of the solution of the following FBSDEs in initial $\xi\in\lr^2_{\f_t}$, with an input $\left(I^b,I^f,I^g\right)\in\lr^2_\f(0,T)\times\lr^2_\f(0,T)\times \lr^2_{\f_T}$:
\begin{equation}\label{FBSDEs_g}
\left\{
\begin{aligned}
    &X_s=\xi+\int_t^s \left[\gamma D_p H(r,X_r,\lr(X_r),P_r,Q_r) +I^b_r\right] dr+ \int_t^s \sigma dB_r,\\
    &P_s=\gamma D_x g\left(X_T,\lr(X_T)\right)+I^g+\int_s^T \left[\gamma D_x H(r,X_r,\lr(X_r),P_r,Q_r)+I^f_r\right] dr-\int_s^T Q_rdB_r,
\end{aligned}
\right.
\end{equation}
where $\gamma\in[0,1]$. For any two initials $\xi^1$ and $\xi^2$ and any two inputs $((I^b)^1,(I^f)^1,(I^g)^1)$ and $((I^b)^2,(I^f)^2,(I^g)^2)$, we denote by $(X^1,P^1,Q^1,v^1)$ and $(X^2,P^2,Q^2,v^2)$ the corresponding solutions of FBSDEs \eqref{FBSDEs_g}, and denote by $\de X:=X'-X$ (we also adopt similar notations for all other differences). Then, from \eqref{small_mf_eff_9}, we know that
\begin{align*}
    &\e\left[\gamma\left[ D_x g\left(X^2_T,\lr(X^2_T)\right)-D_x g\left(X^1_T,\lr(X^1_T)\right)\right]^\top \de X_T  +\left(\de I^g\right)^\top \de X_T - \left(\de P_t\right)^\top \de\xi \right]\\
    \le\ &\e \int_t^T \gamma\left[-\lambda_v|\de v_s|^2 - \left(2\lambda_x-\lambda_m-L_x-\frac{(L_v+3l_x)^2}{4\lambda_v}\right)|\de X_s|^2\right]ds\\
    &+ \e\int_t^T \left( |\de P_s|\cdot |\de I^b_s| + |\de X_s|\cdot |\de I^f_s|\right) ds.
\end{align*}
Then, from \eqref{g_generic_condition}, we know that 
\begin{align}
    &\gamma \e \int_t^T \left[\lambda_v|\de v_s|^2 + \left(2\lambda_x-\lambda_m-L_x-\frac{(L_v+3l_x)^2}{4\lambda_v}\right)|\de X_s|^2\right]ds \notag \\
    \le\ &\e\left[\gamma l_g|\de X_T|^2 + |\de P_t|\cdot |\de\xi|+|\de X_T|\cdot |\de I^g|+\int_t^T \left( |\de P_s|\cdot |\de I^b_s| + |\de X_s|\cdot |\de I^f_s|\right) ds \right]. \label{g_generic_1}
\end{align}
By applying It\^o's formula on $|\de X_s|^2$ and then using Young's inequality, we can compute that
\begin{align*}
    \e\left[|\de X_T|^2\right]=\ & \e\left[|\de\xi|^2+2\gamma\int_t^T \left[(\de X_s)^\top b_1(s)\de X_s+(\de X_s)^\top b_2(s)\de v_s \right] ds+2\int_t^T |\de I^b_s|\cdot |\de X_s| ds\right]\\
    \le\ & \e\left[|\de\xi|^2+2\gamma L\int_t^T \left( |\de X_s|^2 +|\de X_s|\cdot |\de v_s|\right) ds +2\int_t^T |\de I^b_s|\cdot |\de X_s| ds \right]\\
    \le\ & \e \left[|\de\xi|^2+\gamma A\int_t^T \left[ \lambda_v|\de v_s|^2 +\left(2\lambda_x-\lambda_m-L_x-\frac{(L_v+3l_x)^2}{4\lambda_v}\right)|\de X_s|^2\right] ds\right]
\end{align*}
where 
\begin{align*}
    A:=\frac{L\lambda_v+L\sqrt{\lambda_v^2+\lambda_v\left(2\lambda_x-\lambda_m-L_x-\frac{(L_v+3l_x)^2}{4\lambda_v}\right)}}{\lambda_v \left(2\lambda_x-\lambda_m-L_x-\frac{(L_v+3l_x)^2}{4\lambda_v}\right)}.
\end{align*}
Then, by substituting \eqref{g_generic_1} into above inequality, we know that
\begin{align*}
    (1-\gamma Al_g)\e\left[|\de X_T|^2\right]\le (A+2)\e\bigg[ &|\de\xi|^2+|\de P_t|\cdot |\de\xi|+|\de X_T|\cdot |I^g|\\
    &+\int_t^T \left( |\de P_s|\cdot |\de I^b_s| + |\de X_s|\cdot |\de I^f_s|+|\de X_s|\cdot |\de I^b_s|\right) ds \bigg].
\end{align*}
For if 
\begin{align}\label{condition_new}
    Al_g<1,
\end{align}
by noting that $\gamma\in[0,1]$, we have 
\begin{align*}
    \e\left[|\de X_T|^2\right]\le\ C\e\bigg[ & |\de\xi|^2+|\de P_t|\cdot |\de\xi|+|\de X_T|\cdot |I^g|\\
    &+\int_t^T \left( |\de P_s|\cdot |\de I^b_s| + |\de X_s|\cdot |\de I^f_s|\right) ds \bigg],
\end{align*}
where $C$ is a constant depending only on $(L,\lambda_x,\lambda_v,\lambda_m,L_x,L_v,l_x,l_g)$ and it is independent of $\gamma$. Substituting the last estimate back into \eqref{g_generic_1}, we know that for any $\epsilon>0$,
\begin{align}
    &\gamma \e \int_t^T |\de v_s|^2 ds \notag \\
    \le\ & C \e\left[|\de\xi|^2 + |\de P_t|\cdot |\de\xi|+|\de X_T|\cdot |\de I^g|+\int_t^T \left( |\de P_s|\cdot |\de I^b_s| + |\de X_s|\cdot |\de I^f_s|\right) ds \right]\notag\\
    \le\ & \epsilon\e\bigg[\sup_{t\le s\le T} |\de X_s|^2+\sup_{t\le s\le T} |\de P_s|^2 \bigg] + C\left(1+\frac{1}{\epsilon}\right)\e\left[|\de\xi|^2+|\de I^g|^2+  \int_t^T \left( |\de I^b_s|^2 + |\de I^f_s|^2 \right) ds\right]. \label{g_generic_2}
\end{align}
With standard estimates for SDEs, we know that
\begin{align}\label{g_generic_3}
    \e\bigg[\sup_{t\le s\le T} |\de X_s|^2\bigg]\le C(L,T) \e\left[|\de\xi|^2+\int_t^T \left(\gamma|\de v_s|^2+|\de I^b_s|^2 \right)ds\right]
\end{align}
and then, with \eqref{g_generic_3} and standard estimates for BSDEs, 
\begin{align}\label{g_generic_4}
    \e\bigg[\sup_{t\le s\le T} |\de P_s|^2+\int_t^T |\de Q_s|^2ds\bigg]\le C(L,T) \e\left[|\de I^g|^2+\int_t^T \left(\gamma|\de v_s|^2+|\de I^f_s|^2 \right)ds\right].
\end{align}
Subsittuting \eqref{g_generic_3} and \eqref{g_generic_4} into \eqref{g_generic_2}, we can see that by choosing $\epsilon$ small enough, 
\begin{align}
    \gamma\e \int_t^T |\de v_s|^2 ds \le C \e\left[|\de\xi|^2+|\de I^g|^2+  \int_t^T \left( |\de I^b_s|^2 + |\de I^f_s|^2 \right) ds\right]. \label{g_generic_5}
\end{align}
From \eqref{g_generic_3}-\eqref{g_generic_5}, we obatin that
\begin{align}
    &\e\bigg[\sup_{t\le s\le T} |\de X_s|^2+\sup_{t\le s\le T} |\de P_s|^2 +\int_t^T |\de Q_s|^2ds\bigg] \notag \\
    \le\ & C \e\left[|\de\xi|^2+|\de I^g|^2+  \int_t^T \left( |\de I^b_s|^2 + |\de I^f_s|^2 \right) ds\right],  \label{g_generic_6}
\end{align}
where the constant $C$ is independent of $\gamma$. In view of the estimate \eqref{g_generic_6}, we can recursively enlarge the parameter $\gamma$ in FBSDEs \eqref{FBSDEs_g} by using the Banach fixed point theorem; since $C$ is independent of $\gamma$, we can obtain the global solvability. This approach is typicial in the method of continuation in coefficients proposed in \cite{YH2}. The above sketchy argument will work as long as the condition \eqref{condition_new} holds, this in turn allows a flexibility to free $g$ from being convex, indeed; we also refer to \cite[Theorem 7.1]{Anti_mono} for similar conditions. We shall provide more details in our future work.

\section{Solvability of MFTC problem}\label{sec:MFTC}

We  consider the solvability of MFTC problem \eqref{intro_MFTC}. As in Section~\ref{sec:MFG}, we first give a sufficient condition of maximum principle for the MFTC problem \eqref{intro_MFTC}, and then show that the convexity assumption on $f$ is actually the $\beta$-monotonicity condition (Condition~\ref{Condition_mono}(i)) for the well-posedness of FBSDEs \eqref{FB:4} in view of the $\beta$-monotonicity in Condition~\ref{Condition_mono}.

\subsection{Maximum principle}

We take the following assumptions on coefficients.

\textbf{(B1)} The functions $b$ and $\sigma$ are linear. That is,
\begin{align*}
	&b(s,x,m,v)=b_0(s)+b_1(s)x+b_2(s)v+b_3(s)\int_\brn y\;m(dy),\\
	&\sigma(s,x,m,v)=\sigma_0(s)+\sigma_1(s)x+\sigma_2(s)v+\sigma_3(s)\int_\brn y\;m(dy),
\end{align*}
with all the norms of $b_0(s),b_1(s),b_2(s),b_3(s),\sigma_0(s),\sigma_1(s),\sigma_2(s),\sigma_3(s)$ being bounded by $L$.

\textbf{(B2)} The functions $f$ and $g$ have a  quadratic growth \eqref{quadratic_growth}. The derivatives $D_x f$, $D_v f$, $D_y\frac{df}{d\nu}$, $D_x g$ and $D_y\frac{dg}{d\nu}$ exist, and they are continuous in all of their own arguments, and they satisfy
\begin{align*}
    \left|(D_x,D_v) f(s,x',m',v')-(D_x,D_v) f(s,x,m,v)\right|\le\ & L\left(|x'-x|+|v'-v|+W_2(m,m')\right),\\
    \left|D_x g(x',m')-D_x g(x',m')\right|\le\ & L\left(|x'-x|+ W_2(m,m')\right),
\end{align*}
and for any square-integrable random variables $\xi$ and $\xi'$ on the same probability space,
\begin{align*}
    \e\left[\left|D_y \frac{df}{d\nu}\left(s,x',\lr(\xi'),v'\right)\left(\xi'\right)-D_y \frac{df}{d\nu}(s,x',\lr(\xi),v')(\xi)\right|^2\right]\le\ & L\left(|x'-x|^2+|v'-v|^2+ \left\|\xi'-\xi\right\|_2^2 \right),\\
    \e\left[\left|D_y \frac{dg}{d\nu}\left(x',\lr(\xi')\right)\left(\xi'\right)-D_y \frac{dg}{d\nu}(x',\lr(\xi))(\xi)\right|^2\right]\le\ & L\left(|x'-x|^2+ \left\|\xi'-\xi\right\|_2^2 \right).
\end{align*}

\textbf{(B3)} The terminal cost functional $g$ is convex, that is,
\begin{align}\label{convex_g_mftc}
    g\left(x',\lr(\xi')\right)-g(x,\lr(\xi))\geq \left(D_x g (x,\lr(\xi))\right)^\top \left(x'-x\right)+\e\bigg[\left(D_y\frac{d g}{d\nu}(x,\lr(\xi))(\xi)\right)^\top (\xi'-\xi)\bigg].
\end{align}
And there exists $\lambda> 0$ such that for any square-integrable random variables $\xi$ and $\xi'$ on the same probability space,
\begin{align*}
	f\left(s,x',\lr(\xi'),v'\right)-f(s,x,\lr(\xi),v)\geq \ & \left[\begin{pmatrix}  D_x f\\ D_v f \end {pmatrix}(s,x,\lr(\xi),v)\right]^\top \begin{pmatrix}  x'-x\\ v'-v\end {pmatrix} +\lambda \left|v'-v\right|^2 \\
	&+\e\left[\left(D_y\frac{d f}{d\nu}(s,x,\lr(\xi),v)(\xi)\right)^\top \left(\xi'-\xi\right)\right].
\end{align*}

Under above assumptions, we have the following sufficient maximum principle for MFTC \eqref{intro_MFTC}.

\begin{theorem}\label{lem:MP2}
	Under Assumptions (B1)-(B3), suppose that FBSDEs \eqref{FB:4} has a solution $(X,P,Q)\in \sr^2_\f(t,T)\times\lr^2_{\f}(t,T)\times\left(\lr^2_{\f}(t,T)\right)^n$. Then, MFTC problem \eqref{intro_MFTC} has a unique optimal control $v_s:=\hv\left(s,X_s,\lr(X_s),P_s,Q_s\right),\  s\in[t,T]$, where the map $\hv$ is defined in \eqref{hv}.
\end{theorem}

The proof of Theorem~\ref{lem:MP2} is similar to that in \cite{AB8}, and is thus omitted here.

\subsection{Well-posedness of the forward-backward system}

We shall use Lemma~\ref{lem:1} to give the well-posedness of our FBSDEs \eqref{FB:4}, which can be viewed as a subcase of FBSDEs \eqref{FB:11} by setting
\begin{equation}\label{def:BAFG'}
	\begin{split}
		&\mathbf{B}(s,X,P,Q):=D_p H(s,X,\lr(X),P,Q),\\
		&\mathbf{A}(s,X,P,Q):=D_q H(s,X,\lr(X),P,Q),\\
		&\mathbf{F}(s,X,P,Q):=-D_x H(s,X,\lr(X),P,Q)-\widetilde{\e}\left[D_y \frac{d H}{d\nu}\left(s,\widetilde{X},\lr(X),\widetilde{P},\widetilde{Q}\right)(X) \right],\\
		&\mathbf{G}(X):=D_xg(X,\lr(X)) +\widetilde{\e}\left[D_y\frac{d g}{d\nu}\left(\widetilde{X},\lr(X)\right)(X)\right],
	\end{split}
\end{equation}
for $X,P\in L^2\left(\Omega,\f,\mathbb{P};\brn\right)$ and $Q\in L^2\left(\Omega,\f,\mathbb{P};\br^{n\times n}\right)$, where $\widetilde{X}$, $\widetilde{P}$ and $\widetilde{Q}$ are respectively independent copies of $X$, $P$ and $Q$. We now show that Assumptions (B1)-(B3) can ensure the $\beta$-monotonicity in Condition~\ref{Condition_mono}, and then imply with the global well-posedness of FBSDEs \eqref{FB:4}. The choices of $\mathbf{F}$ and $\mathbf{G}$ in \eqref{def:BAFG'} are different from that in \eqref{def:BAFG} due to the arithmetic difference between the backward equation in FBSDEs \eqref{FB:4} and in FBSDEs \eqref{intro_2} respectively, which is because the state process of the MFTC problem \eqref{intro_MFTC} depends simultaneously on the distribution of the current controlled state, while the state process of MFG \eqref{intro_MFG} depends on the the equilibrium distribution from the population.

\begin{theorem}
    Under Assumptions (B1)-(B3), coefficients in \eqref{def:BAFG'} satisfy Condition~\ref{Condition_mono} with the map $\beta(s,X,P,Q):=\hv(s,X,\lr(X),P,Q)$ and the parameter $\Gamma_\beta=0$. As a consequence, FBSDEs \eqref{FB:4} have a unique adapted solution.
\end{theorem}

\begin{proof}
Here, we check Condition~\ref{Condition_mono}(i). From \eqref{def:BAFG'}, Assumption (B1) and the definition of $\hv(\cdot)$ in \eqref{hv}, we can write
\begin{align*}
    \mathbf{B}(s,X,P,Q)=\ & b_0(s)+b_1(s)X+b_2(s)\hv+b_3(s)\e[X],\\
	\mathbf{A}^j(s,X,P,Q)=\ & \sigma_0^j(s)+\sigma_1^j(s)X+\sigma_2^j(s)\hv+\sigma_3^j(s)\e[X],\\
	\mathbf{F}(s,X,P,Q)=\ & -b_1(s)^\top P-\sum_{j=1}^n \sigma_1^j(s)^\top Q^j-D_xf(s,X,\lr(X),\hv) \\
    &-b_3(s)^\top\e[P]-\sum_{j=1}^n \sigma_3^j(s)^\top \e\left[Q^j\right]-\widetilde{\e}\left[D_y \frac{d f}{d\nu}\left(s,\widetilde{X},\lr(X),\widetilde{\hv}\right)(X) \right],
\end{align*}
where $\hv:=\hv(s,X,\lr(X),P,Q)$ and $\widetilde{\hv}$ is an independent copies of $\hv$; the map $\hv(\cdot)$ is defined in \eqref{hv}. Then, we can compute that
\begin{align}
    &\e\bigg[\left(\mathbf{F}(s,X',P',Q')-\mathbf{F}(s,X,P,Q)\right)^\top  (X'-X)+\left(\mathbf{B}(s,X',P',Q')-\mathbf{B}(s,X,P,Q)\right)^\top  (P'-P) \notag \\
	&\quad +\sum_{j=1}^n \left(\mathbf{A}^j(s,X',P',Q')-\mathbf{A}^j(s,X,P,Q)\right)^\top  \left({Q'}^{j}-Q^j\right)\bigg] \notag \\
	\ =& -\e\Bigg\{ \Bigg[\begin{pmatrix}
	    D_x f \\ D_v f
	\end{pmatrix}(s,\cdot)\Bigg|^{\left(X',\lr(X'),\hv'\right)}_{\left(X,\lr(X),\hv\right)} \Bigg]^\top \begin{pmatrix}
	    X'-X\\ \hv'-\hv
	\end{pmatrix} \notag \\
    &\qquad +\widetilde{\e}\left[D_y \frac{d f}{d\nu}\left(s,\widetilde{X'},\lr(X'),\widetilde{\hv'}\right)(X')-D_y \frac{d f}{d\nu}\left(s,\widetilde{X},\lr(X),\widetilde{\hv}\right)(X) \right]^\top(X'-X) \Bigg\} . \label{lem:mono_mftc_1}
\end{align}
By Fubini's lemma, we know that
\begin{align*}
    &\e\left\{\widetilde{\e}\left[D_y \frac{d f}{d\nu}\left(s,\widetilde{X'},\lr(X'),\widetilde{\hv'}\right)(X')-D_y \frac{d f}{d\nu}\left(s,\widetilde{X},\lr(X),\widetilde{\hv}\right)(X) \right]^\top(X'-X) \right\}\\
    =\ & \e\left\{ \widetilde{\e}\left[\left(D_y \frac{d f}{d\nu}\left(s,{X'},\lr(X'),{\hv'}\right)\left(\widetilde{X'}\right)-D_y \frac{d f}{d\nu}\left(s,X,\lr(X),\hv\right)\left(\widetilde{X}\right) \right)^\top \left(\widetilde{X'}-\widetilde{X}\right) \right] \right\},
\end{align*}
then, from \eqref{lem:mono_mftc_1} and the convexity of $f$ in Assumption (B3), we know that Condition \ref{Condition_mono} (i)(a) is valid with the map $\beta(s,X,P,Q):=\hv(s,X,\lr(X),P,Q)$ by setting $\Lambda_\beta=2\lambda$ and $\Gamma_\beta=0$. In a similar way, we can show that Condition \ref{Condition_mono} (i)(b) is also satisfied.
\end{proof}

\section{The case for generic drift}\label{sec:generic}

So far we mainly focus on the case when $b$ is linear in $x$ and $v$, we next proceed on the more general seetting such that $b$ can be non-linear in $x$, $v$ and $m$. We study both MFG \eqref{intro_2} and the MFTC problem \eqref{intro_MFTC} with a generic $b$.

\subsection{Mean field games with generic drifts}\label{subsec:mfg_gene}

In this subsection, we shall give the solvability of MFG \eqref{intro_MFG} without Assumption (A1); instead, we adopt the following assumptions on $b$ and $\sigma$.

\textbf{(A1')} The coefficient $b$ grows linearly and is $L$-lipschitz continuous in $(x,v)$ and $l_m$-lipschitz continuous in $m$; it is continuously differentiable in $x$ and $v$, with the derivatives $D_x b$ and $D_v b$ being bounded by $L$. Moreover, for any $s\in[0,T]$, $x,x'\in\brn$,  $v,v'\in\brd$ and $m,m'\in \pr_2(\brn)$,
\begin{equation}\label{generic:condition:b}
\begin{aligned}
    &\left|(D_x b,D_v b)(s,x',m',v')- (D_x b,D_v b)(s,x,m,v)\right|\\
    \le\ & \frac{L_b^x|x'-x|+L_b^v|v'-v|+L_b^m W_2(m,m')}{1+|x|\vee |x'|+ |v|\vee |v'|+W_2(m,\delta_0)\vee W_2(m',\delta_0) };
\end{aligned}
\end{equation}
and there exists $\lambda_b >0$ such that,
\begin{align}\label{generic:condition:b'}
    (D_v b)(D_v b)^\top(s,x,m,v) \geq \lambda_b I_n,\quad \forall (s,x,v)\in [0,T]\times \brn\times \brd.
\end{align}
The coefficient $\sigma$ is linear in $x$ as $\sigma(s,x)=\sigma_0(s)+\sigma_1(s)x$, with the the norms of matrices $\sigma_0$ and $\sigma_1$ being bounded  by $L$. 

In the condition \eqref{generic:condition:b}, if the coefficient $b$ demanded is linear in $x$ and $v$, then $L_b^x=L_b^v=L_b^m=0$. Therefore, (A1') extends the linear assumption on $b$ in Assumption (A1). Our Assumption (A1') is also used in \cite{AB10',AB10''} for the solvability of first-order MFGs and MFTC problems with a generic $b$. We also refer to \cite{GW} for similar conditions to obtain the solvability of mean field game master equation with an analytical approach, but on the Hamiltonian functional $H$, not the individual assumptions on the drift functional $b$ and on the cost functionals. In view of Assumption (A1'), the system of FBSDEs associated with MFG \eqref{intro_MFG} is as follows:
\begin{equation}\label{FB:mfg_generic}
\left\{
    \begin{aligned}
        &X_s = \xi+\int_t^s b(r,X_r,\lr(X_r),v_r)dr+ \int_t^s \left[\sigma_0(r)+\sigma_1(r)X_r\right]dB_r,\\
        &P_s = -\int_s^T Q_r dB_r+D_x g(X_T,\lr(X_T))\\
        &\ \qquad +\int_s^T \bigg[ D_x b(r,X_r,\lr(X_r),v_r)^\top P_r+\sum_{j=1}^n \left(\sigma^j_1(r)\right)^\top Q_r^j + D_x f(r,X_r,\lr(X_r),v_r)\bigg]dr,
    \end{aligned}
\right.
\end{equation}
subject to
\begin{equation}\label{FB:mfg_generic_condition}
    D_v b(s,X_s,\lr(X_s),v_s)^\top P_s+ D_v f(s,X_s,\lr(X_s),v_s)=0,\quad s\in[t,T].
\end{equation}

We first give the corresponding maximum principle for MFG \eqref{intro_MFG} with Assumption (A1').

\begin{theorem}\label{lem:MP1'}
	Under Assumptions (A1'), (A2) and (A3), suppose that $f$ satisfies the convexity \eqref{strong_convex_1} with constants $\lambda_x\geq \frac{L^2L_b^x}{\lambda_b}$ and $\lambda_v>\frac{L^2L_b^v}{\lambda_b}$. If FBSDEs \eqref{FB:mfg_generic}-\eqref{FB:mfg_generic_condition} have a solution $\left(\bar{X},\bar{P},\bar{Q},\bar{v}\right)\in \sr^2_\f(t,T)\times\sr^2_{\f}(t,T)\times\left(\lr^2_{\f}(t,T)\right)^n\times\lr^2_{\f}(t,T)$, then $\bar{v}$ is the unique solution of MFG \eqref{intro_MFG}.
\end{theorem}

\begin{proof}
We now prove the coercive property for $J$. For any control $v\in \lr_{\f}^2(t,T)$, we denote by $X^v$ the corresponding state. From Assumptions (A2) and the convexity of $f$ and $g$, we have
\begin{align}
    &J\left(v;\lr(\bar{X}_s),t\le s\le T\right)-J\left(\bar{v};\lr(\bar{X}_s),t\le s\le T\right) \notag \\
    \geq\  & \e\Bigg\{\int_t^T \bigg[ \left[\begin{pmatrix}
        D_x f\\ D_v f
    \end{pmatrix} \left(s,\bar{X}_s,\lr\left(\bar{X}_s\right),\bar{v}_s\right)\right]^\top \begin{pmatrix}
        X^v_s-\bar{X}_s\\ v_s-\bar{v}_s
    \end{pmatrix} +\lambda_x\left|X^v_s-\bar{X}_s\right|^2+ \lambda_v\left|v_s-\bar{v}_s\right|^2\bigg]ds \notag \\
    &\qquad +D_x g\left(\bar{X}_T\right)^\top \left(X^v_T-\bar{X}_T\right)\Bigg\}. \label{thm:ger_max_5}
\end{align}
By applying It\^o's formula on $\bar{P}_s^\top \left(X^v_s-\bar{X}_s\right)$ and taking expectation, we have
\begin{align}
    &\e\left[ D_x g\left(\bar{X}_T\right)^\top \left(X^v_T-\bar{X}_T\right)\right] \notag \\
    =\ & \e\Bigg[\int_t^T \bar{P}_s^\top \left[b\left(s,X^v_s,\lr\left(\bar{X}_s\right),v_s\right)-b\left(s,\bar{X}_s,\lr\left(\bar{X}_s\right),\bar{v}_s\right)\right] + \sum_{j=1}^n \left(\bar{Q}^j_s\right)^\top\sigma_1^j(s)\left(X^v_s-\bar{X}_s\right)  \notag \\
    &\quad\qquad -\bigg[ D_x b\left(s,\bar{X}_s,\lr\left(\bar{X}_s\right),\bar{v}_s\right)^\top \bar{P}_s+\sum_{j=1}^n \left(\sigma_1^j(s)\right)^\top \bar{Q}^j_s+D_x f\left(s,\bar{X}_s,\lr\left(\bar{X}_s\right),\bar{v}_s\right) \bigg]^\top \left( X^v_s-\bar{X}_s\right) ds\Bigg] \notag \\
    =\ & \e\Bigg[\int_t^T \left[b\left(s,X^v_s,\lr\left(\bar{X}_s\right),v_s\right)-b\left(s,\bar{X}_s,\lr\left(\bar{X}_s\right),\bar{v}_s\right)-D_x b\left(s,\bar{X}_s,\bar{v}_s\right) \left( X^v_s-\bar{X}_s\right)  \right]^\top \bar{P}_s \notag \\
    &\quad\qquad - D_x f\left(s,\bar{X}_s,\lr\left(\bar{X}_s\right),\bar{v}_s\right)^\top \left( X^v_s-\bar{X}_s\right) ds\Bigg]. \notag
\end{align}
From the first order condition \eqref{FB:mfg_generic_condition}, we know that 
\begin{align}\label{thm:ger_max_7}
    D_v b\left(s,\bar{X}_s,\lr\left(\bar{X}_s\right),\bar{v}_s\right)^\top \bar{P}_s+D_vf\left(s,\bar{X}_s,\lr\left(\bar{X}_s\right),\bar{v}_s\right)=0,
\end{align}
and therefore, 
\begin{align}
    &\e\left[ D_x g\left(\bar{X}_T\right)^\top \left(X^v_T-\bar{X}_T\right)\right] \notag \\
    =\ & \e\Bigg[\int_t^T \bigg[b(s, \cdot,\lr\left(\bar{X}_s\right),\cdot)\bigg|^{\left(X^v_s,v_s\right)}_{\left(\bar{X}_s,\bar{v}_s\right)} - \left[(D_x b, D_v b) \left(s,\bar{X}_s,\lr\left(\bar{X}_s\right),\bar{v}_s\right)\right] \begin{pmatrix} X^v_s-\bar{X}_s\\  v_s-\bar{v}_s\end{pmatrix}\bigg]^\top \bar{P}_s \notag \\
    &\qquad\qquad - \left[\begin{pmatrix}
        D_x f\\ D_v f
    \end{pmatrix} \left(s,\bar{X}_s,\lr\left(\bar{X}_s\right),\bar{v}_s\right)\right]^\top \begin{pmatrix} X^v_s-\bar{X}_s\\
    v_s-\bar{v}_s\end{pmatrix} ds\Bigg]. \label{thm:ger_max_4}
\end{align}
In view of \eqref{thm:ger_max_7}, the condition \eqref{generic:condition:b'} and the \textit{Schur complement}, we can compute that 
\begin{align*}
    \bar{P}_s=-\left[\left((D_vb)^\top (D_vb)\right)^{-1}(D_vb)^\top \left(s,\bar{X}_s,\lr\left(\bar{X}_s\right),\bar{v}_s\right)\right](D_v f)\left(s,\bar{X}_s,\lr\left(\bar{X}_s\right),\bar{v}_s\right),
\end{align*}
(also see the following Remark~\ref{remark:cone}), and then, from (A2), we know that
\begin{align}\label{p_norm}
    \left|\bar{P}_s\right|\le \frac{L^2}{\lambda_b}\left[1+\left|\bar{X}_s\right|+W_2\left(\lr\left(\bar{X}_s\right),\delta_0\right)+\left|\bar{v}_s\right|\right].
\end{align} 
Therefore, we deduce from the condition \eqref{generic:condition:b} that 
\begin{align}
    & \left|\e\left[\int_t^T \left[b(s, \cdot,\lr\left(\bar{X}_s\right),\cdot)\Big|^{\left(X^v_s,v_s\right)}_{\left(\bar{X}_s,\bar{v}_s\right)} -\left[(D_x b, D_v b) \left(s,\bar{X}_s,\lr\left(\bar{X}_s\right),\bar{v}_s\right)\right] \begin{pmatrix} X^v_s-\bar{X}_s\\  v_s-\bar{v}_s\end{pmatrix} \right]^\top \bar{P}_s \right]\right| \notag \\
    \le\ & \frac{L^2}{\lambda_b}\ \e\left[\int_t^T  \left(L_b^x \left|X^v_s-\bar{X}_s\right|^2 +L_b^v \left|v_s-\bar{v}_s\right|^2 \right) ds\right]. \label{thm:ger_max_6}
\end{align}
Substituting \eqref{thm:ger_max_6} into \eqref{thm:ger_max_4}, we know that
\begin{align*}
    &\e\left\{ \int_t^T \left[\begin{pmatrix}
        D_x f \\ D_v f
    \end{pmatrix}\left(s,\bar{X}_s,\lr\left(\bar{X}_s\right),\bar{v}_s\right)\right]^\top \begin{pmatrix}
        X^v_s-\bar{X}_s\\ v_s-\bar{v}_s
    \end{pmatrix} ds +D_x g\left(\bar{X}_T\right)^\top \left(X^v_T-\bar{X}_T\right)\right\}  \\
    \geq\ & -\frac{L^2}{\lambda_b}\ \e\left[\int_t^T  \left(L_b^x \left|X^v_s-\bar{X}_s\right|^2 +L_b^v \left|v_s-\bar{v}_s\right|^2 \right) ds\right].
\end{align*}
Combining the last inequality and \eqref{thm:ger_max_5}, we know that 
\begin{align*}
    J(v)-J\left(\bar{v}\right) \geq\  & \e\left\{\int_t^T \left[ \left(\lambda_x-\frac{L^2L_b^x}{\lambda_b}\right) \left|X^v_s-\bar{X}_s\right|^2+ \left(\lambda_v-\frac{L^2L_b^v}{\lambda_b}\right)\left|v_s-\bar{v}_s\right|^2\right]ds\right\}\\
    \geq\ & \left(\lambda_v-\frac{L^2L_b^v}{\lambda_b}\right) \e\left[\int_t^T \left|v_s-\bar{v}_s\right|^2ds\right],
\end{align*}
from which we obtain the claimed coerciveness.
\end{proof}

\begin{remark}\label{remark:cone}
    In the proof of Theorem~\ref{lem:MP1'}, the condition \eqref{p_norm} is also known as a cone property (cone condition) which was first proposed in \cite{AB10',AB10''} for the study on the first-order mean field theory. We would like to emphasize that the cone condition is satisfied under Assumptions (A1'), (A2) and (A3), which does not require additional assumption. We now show that Assumption (A1') and Condition~\ref{assumption_mono_1} can ensure the $\beta$-monotonicity corresponding to FBSDEs \eqref{FB:mfg_generic}, and then guarantee the global well-posedness of FBSDEs \eqref{intro_2}.
\end{remark}

\begin{theorem}\label{thm:mono_generic}
    Under Assumptions (A1'), (A2) and (A3), suppose that the terminal cost functional $g$ satisfies the displacement monotonicity condition \eqref{g_mono}, and the running cost functional $f$ satisfies Condition~\ref{assumption_mono_1}. Then, when $\lambda_v>\frac{2L^2L_b^v}{\lambda_b}$ and 
    \begin{equation}\label{thm:mono_generic:condition}
    \begin{aligned}
        2\lambda_x-\lambda_m\geq\ & L_x+\frac{L^2 (2l_m+2L_b^x+L_b^m)\lambda_b+3L^3(L_b^x+L_b^m) l_m}{\lambda_b^2}\\
        &+\frac{1}{4\lambda_v}\left(L_v+3l_x+\frac{L^2(l_m+L_b^m)\lambda_b+3L^3L_b^v l_m}{\lambda_b^2}\right)^2,
    \end{aligned}
    \end{equation}
    FBSDEs \eqref{intro_2} have a unique global solution.
\end{theorem}

\begin{proof}
We shall show that the coefficients of FBSDEs \eqref{FB:mfg_generic} satisfy Condition~\ref{Condition_mono}. For $s\in[t,T]$, $X,X',P,P'\in L^2\left(\Omega,\f,\mathbb{P};\brn\right)$, $Q,Q'\in L^2\left(\Omega,\f,\mathbb{P};\br^{n\times n}\right)$ and $V,V'\in L^2\left(\Omega,\f,\mathbb{P};\brd\right)$ respectively satisfying 
\begin{equation}\label{thm_generic_5}
\begin{aligned}
    &D_v b(s,X,\lr(X),V)^\top P+ D_v f(s,X,\lr(X),V)=0,\\
    &D_v b(s,X',\lr(X'),V')^\top P+ D_v f(s,X',\lr(X'),V')=0,
\end{aligned}
\end{equation}
we can compute that
\begin{align}
	&\e\bigg\{\left(b(s,X',\lr(X'),V')-b(s,X,\lr(X),V)\right)^\top  (P'-P)+\sum_{j=1}^n \left(\sigma^j_1(s)(X'-X)\right)^\top  \left({Q'}^{j}-Q^j\right) \notag \\
    &\quad - \bigg[D_x b(s,X',\lr(X'),V')^\top P'+\sum_{j=1}^n \left(\sigma^j_1(s)\right)^\top {Q'}^j + D_x f(s,X',\lr(X'),V') \notag \\
    &\quad\qquad -D_x b(s,X,\lr(X),V)^\top P+\sum_{j=1}^n \left(\sigma^j_1(s)\right)^\top Q^j + D_x f(s,X,\lr(X),V)\bigg]^\top (X'-X)\bigg\} \notag \\
    =\ & \e\bigg\{-\left[b(s,X,\lr(X),V)-b(s,X',\lr(X'),V') -D_x b(s,X',\lr(X'),V') (X-X')\right]^\top  (P'-P)\notag \\
    &\ \quad - \left[ \left(D_x b(s,X',\lr(X'),V') -D_x b(s,X,\lr(X),V)\right)(X'-X) \right]^\top P  \notag \\
    &\ \quad - \left[D_xf(s,X',\lr(X'),V')-D_xf(s,X,\lr(X),V)\right]^\top  (X'-X) \bigg\} \notag\\
    =\ & \e\Bigg\{ \left[b(s,X,\lr(X'),V)-b(s,X,\lr(X),V)\right]^\top (P'-P) \notag \\
    &\ \quad +\left[b(s,X',\lr(X),V')+b(s,X,\lr(X'),V)-b(s,X,\lr(X),V)-b(s,X',\lr(X'),V')\right]^\top P \notag \\
    &\ \quad -\bigg[b(s,\cdot,\lr(X'),\cdot)\bigg|^{\left(X,V\right)}_{\left(X',V'\right)} - \left[(D_x b, D_v b)\left(s,X',\lr(X'),V'\right)\right] \begin{pmatrix}
        X-X'\\ V-V'
    \end{pmatrix} \bigg]^\top P' \notag \\
    &\ \quad -\bigg[b(s,\cdot,\lr(X),\cdot)\bigg|^{\left(X',V'\right)}_{\left(X,V\right)} - \left[(D_x b, D_v b)\left(s,X,\lr(X),V\right)\right] \begin{pmatrix}
        X'-X\\ V'-V
    \end{pmatrix} \bigg]^\top P \notag \\
    &\ \quad - \bigg[\begin{pmatrix}
        D_x f\\ D_v f
    \end{pmatrix}(s,\cdot)\bigg|^{\left(X',\lr(X')),V'\right)}_{\left(X,\lr(X)),V\right)} \bigg]^\top \begin{pmatrix}
        X'-X\\ V'-V
    \end{pmatrix} \Bigg\}. \label{thm_generic_1}
\end{align}    
From \eqref{thm_generic_5} and Assumptions (A1') and (A2), we can compute that 
\begin{align}
    &|P'-P|  \notag \\
    =\ & \Bigg|\left[\left((D_vb)^\top (D_vb)\right)^{-1}(D_vb)^\top \left(s,X',\lr(X'),V'\right)\right](D_v f)(s,\cdot)\Big|^{(X',\lr(X'),V')}_{(X,\lr(X),V)} \notag \\
    &+\left\{\left[\left((D_vb)^\top (D_vb)\right)^{-1}(D_vb)^\top\right] \left(s,\cdot\right)\bigg|^{(X',\lr(X'),V')}_{(X,\lr(X),V)} \right\}(D_v f)\left(s,X,\lr\left(X\right),V\right)\Bigg| \notag \\
    \le\ & \frac{L}{\lambda_b} \left|(D_v f)(s,\cdot)\Big|^{(X',\lr(X'),V')}_{(X,\lr(X),V)}\right|+\frac{3L^2}{\lambda_b^2} \left|(D_v b)(s,\cdot)\Big|^{(X',\lr(X'),V')}_{(X,\lr(X),V)}\right| \cdot \left|(D_v f)\left(s,X,\lr\left(X\right),V\right)\right| \notag \\
    \le\ & \left(\frac{L^2}{\lambda_b}+\frac{3L^3L_b^x}{\lambda_b^2}\right) |X'-X|+ \left(\frac{L^2}{\lambda_b}+\frac{3L^3L_b^v}{\lambda_b^2}\right) |V'-V|+ \left(\frac{L^2}{\lambda_b}+\frac{3L^3L_b^m}{\lambda_b^2}\right) W_2(\lr(X),\lr(X')), \notag
\end{align}
and therefore, 
\begin{align}
    &\e\left\{ \left[b(s,X,\lr(X'),V)-b(s,X,\lr(X),V)\right]^\top (P'-P)\right\} \notag \\
    \le\ & \left(\frac{2L^2 l_m}{\lambda_b}+\frac{3L^3(L_b^x+L_b^m) l_m}{\lambda_b^2}\right) \|X'-X\|_2^2 + \left(\frac{L^2l_m}{\lambda_b}+\frac{3L^3L_b^v l_m}{\lambda_b^2}\right) \|V'-V\|_2 \cdot \|X'-X\|_2. \label{thm_generic_7}
\end{align}
Similarly, from \eqref{thm_generic_5}, the condition \eqref{generic:condition:b'} and Assumption (A2), we also have the cone properties
\begin{align*}
    |P|\le \frac{L^2}{\lambda_b}\left(1+|X|+W_2(\lr(X),\delta_0)+\left|V\right|\right),\quad |P'|\le \frac{L^2}{\lambda_b}\left(1+|X'|+W_2(\lr(X'),\delta_0)+\left|V'\right|\right),
\end{align*}
and therefore, from the condition \eqref{generic:condition:b}, we know that
\begin{align}
    & \e\Bigg\{ \left[b(s,X',\lr(X),V')+b(s,X,\lr(X'),V)-b(s,X,\lr(X),V)-b(s,X',\lr(X'),V')\right]^\top P \notag \\
    &\ \quad -\left[b(s,\cdot,\lr(X'),\cdot)\Big|^{\left(X,V\right)}_{\left(X',V'\right)} - \left[(D_x b, D_v b)\left(s,X',\lr(X'),V'\right)\right] \begin{pmatrix}
        X-X'\\ V-V'
    \end{pmatrix} \right]^\top P' \notag \\
    &\ \quad -\left[b(s,\cdot,\lr(X),\cdot)\Big|^{\left(X',V'\right)}_{\left(X,V\right)} - \left[(D_x b, D_v b)\left(s,X,\lr(X),V\right)\right] \begin{pmatrix}
        X'-X\\ V'-V
    \end{pmatrix} \right]^\top P  \Bigg\} \notag\\
    \le\ & \frac{L^2 L_b^m}{\lambda_b} \left(\|X'-X\|_2^2+\|V'-V\|_2\cdot \|X'-X\|_2\right)+\frac{2 L^2L_b^x}{\lambda_b}\|X'-X\|_2^2+\frac{2 L^2L_b^v}{\lambda_b}\|V'-V\|_2^2. \label{thm_generic_2}
\end{align}
Since $f$ satisfies Condition~\ref{assumption_mono_1}, from the proof of Theorem~\ref{thm:mono} (also referring to \eqref{small_mf_eff_3'}, \eqref{small_mf_eff_8} and \eqref{small_mf_eff_9}), we know that
\begin{align}
    & \e\Bigg\{ - \Bigg[\begin{pmatrix}
        D_x f\\ D_v f
    \end{pmatrix}(s,\cdot)\Bigg|^{\left(X',\lr(X')),V'\right)}_{\left(X,\lr(X)),V\right)} \Bigg]^\top \begin{pmatrix}
        X'-X\\ V'-V
    \end{pmatrix}  \Bigg\} \notag\\
    \le\ & -2\lambda_v \left\| \hv'-\hv\right\|_2^2 -(2\lambda_x-\lambda_m-L_x) \left\| X'-X \right\|_2^2 +(L_v+3l_x) \left\|X'-X\right\|_2 \cdot \left\|V'-V\right\|_2 .\label{thm_generic_4}
\end{align}
Substituting \eqref{thm_generic_7}-\eqref{thm_generic_4} back into \eqref{thm_generic_1}, from the Young's inequality and Condition \eqref{thm:mono_generic}, we know that 
\begin{align*}
    &\textit{The left hand side of \eqref{thm_generic_1}}\\
    \le\ & -\left(2\lambda_v-\frac{2 L^2L_b^v}{\lambda_b} \right)\left\| V'-V\right\|_2^2+\left(L_v+3l_x+\frac{L^2(l_m+L_b^m)}{\lambda_b}+\frac{3L^3L_b^v l_m}{\lambda_b^2}\right) \left\|X'-X\right\|_2 \cdot \left\|V'-V\right\|_2 \\
    &-\left(2\lambda_x-\lambda_m-L_x-\frac{L^2 (2l_m+2L_b^x+L_b^m)}{\lambda_b}-\frac{3L^3(L_b^x+L_b^m) l_m}{\lambda_b^2} \right) \left\| X'-X \right\|_2^2 \\
    \le\ & -\left(\lambda_v-\frac{2L^2L_b^v}{\lambda_b}\right) \|V'-V\|_2^2,
\end{align*}  
from which we know that Condition \ref{Condition_mono} (i)(a) is valid with $\Lambda_\beta=\lambda_v-\frac{2L^2L_b^v}{\lambda_b}$ and $\Gamma_\beta=0$. Since $g$ satisfies the displacement monotonicity condition, from \eqref{use_latter}, we know that Condition \ref{Condition_mono} (i)(b) is satisfied. As a consequence of Lemma~\ref{lem:1}, we obtain the well-posedness of FBSDEs \eqref{FB:mfg_generic}.
\end{proof}

\begin{remark}
    In the relation \eqref{thm:mono_generic:condition}, the parameters are not the optimal, but we do not drill down into the details in this article. The conditions $\lambda_v>\frac{2L^2L_b^v}{\lambda_b}$ and \eqref{thm:mono_generic} in Theorem~\ref{thm:mono_generic} seem to be restrictive, however, when the coefficient $b$ is linear in $x$ and $v$ and does not functionally depend on $m$, then $L_b^x=L_b^v=L_b^m=l_m=0$, and the conditions above are reduced to that in Condition~\ref{assumption_mono_1}.
\end{remark}

As a direct consequence of Theorems~\ref{lem:MP1'} and \ref{thm:mono_generic}, we obtain the following solvability of MFG \eqref{intro_MFG} with the generic $b$.

\begin{corollary}
    Under assumptions in Theorem~\ref{thm:mono_generic}, MFG \eqref{intro_MFG} has a unique global solution.
\end{corollary}

\subsection{Relation with the mean field game HJB-FP system}

MFG \eqref{intro_MFG} is associated with the following HJB-FP system:
\begin{equation}\label{HJB-FP}
	\left\{
	\begin{aligned}
        &\dd_s v^{t,\mu}(s,x)+H\left(s,x,m^{t,\mu}(s),D_x v^{t,\mu}(s,x),\frac{1}{2}D_x^2 v^{t,\mu}(s,x)\sigma\left(s,x,m^{t,\mu}(s)\right)\right)=0,\quad s\in[t,T),\\
        &\dd_s m^{t,\mu}+ {div}\left[ D_p H \left(s,x,m^{t,\mu}(s),D_x v^{t,\mu}(s,x)\right) m^{t,\mu}\right]\\
		&\:\:\ \qquad -\sum_{i,j=1}^n \dd_{x_i}\dd_{x_j} \left[a_{ij}\left(s,x,m^{t,\mu}(s)\right)m^{t,\mu} \right]=0,\quad s\in(t,T],\\
	    &m^{t,\mu}(t)=\mu,\quad v^{t,\mu}(T,x)=g\left(x,m^{t,\mu}(T)\right),
	\end{aligned}
	\right.
\end{equation}
where $a_{ij}(s,x,m):=\frac{1}{2}\left(\sigma\sigma^\top\right)_{ij}(s,x,m)$. This HJB-FP system is related to our FBSDEs \eqref{intro_2} in the following manner: suppose that $\left(X^{t,\xi},P^{t,\xi},Q^{t,\xi}\right)$ is an adapted solution of FBSDEs \eqref{intro_2} with initial condition $X_t^{t,\xi}=\xi\sim \mu$, then,
\begin{align}\label{relation:FP-HJB-FBSDE}
    m^{t,\mu}(s)=\lr\left(X^{t,\xi}_s\right), \quad D_x v^{t,\mu}\left(s,X^{t,\xi}_s\right)=P^{t,\xi}_s,\quad s\in[t,T].
\end{align}
This relation \eqref{relation:FP-HJB-FBSDE} also gives a decoupling field of FBSDEs \eqref{intro_2}. The HJB-FP system \eqref{HJB-FP} is widely studied in the contemporary literature; see \cite{CP,CDLL,GDA,JM1,JM2,JM3} for instance. In such works,  $\sigma$ is usually assumed to be constant or non-degenerate, and System \eqref{HJB-FP} can be studied via the solution a parabolic equation. For the existence result for System \eqref{HJB-FP} and also the solution regularity, the regularity of  the  functional Hamiltonian $H(s,x,m,p,q)$ is directly assumed. For the uniqueness result for System \eqref{HJB-FP}, some monotonicity conditions are required for $H$ and $g$. The  Lasry-Lions monotonicity condition:
\begin{align*}
    \int_{\brn} \left[g(x,m')-g(x,m) \right] (m'-m)(dx)\geq 0,
\end{align*}
which is introduced by Lasry and Lions \cite{JM1,JM2} and then used in the literature \cite{CDLL,GDA}, is now well-known. Its relation with the displacement monotonicity condition is referred  to \cite{SA,AB9',GW,GM}.

In our work, to see the matter from the control perspective more clearly, we give conditions directly on coefficients $b$, $\sigma$ and $f$, rather than on the induced functional  Hamiltonian $H$ as in the above mentioned literature studying the HJB-FP system \eqref{HJB-FP} via analytical methods. However, we can show that our assumptions on $b$, $\sigma$, $f$ for the control method are related to those on $H$ for the analytical method. For example, we show that the condition $\lambda_v>\frac{2L^2L_b^v}{\lambda_b}$ in Theorem~\ref{thm:mono_generic} can guarantee the concavity for the Hamiltonian functional $H$ in $p$; this concavity assumption is widely used in the existing literature, such as \cite{CP,CDLL,GW,GDA,HZ2}. For convenience, we here prove the particular case when the dimension $n=d=1$ and the coefficients $b$ and $f$ are twice differentiable in $v$, although we do not need coefficients to be twice differentiable just for the sake of the well-posedness  of the MFGs and MFTC problem. 
We have
\begin{align}
    D_{p}^2 H(s,x,m,p,q)=\ & D_p \left[b(s,x,m,\hv(s,x,m,p))\right] \notag \\
    =\ & D_v b(s,x,m,\hv(s,x,m,p)) D_p\hv(s,x,m,p). \label{rk_non_convex_H_1}
\end{align}
From the first order optimality condition, we know that
\begin{align*}
    D_v b(s,x,m,\hv(s,x,m,p)) \ p+ D_v f(s,x,m,\hv(s,x,m,p)) =0,
\end{align*}
and by differentiating $p$ in this last equation, we have
\begin{align*}
    & \left[D_v^2  b(s,x,m,\hv(s,x,m,p)) \ p+ D_v^2  f(s,x,m,\hv(s,x,m,p)) \right] D_p \hv (s,x,m,p)\\
    =\ & -D_v b(s,x,m,\hv(s,x,m,p)). 
\end{align*}
Then, in view of the condition \eqref{generic:condition:b'}, we know that
\begin{align*}
    D_p \hv (s,x,m,p)=\frac{|D_v b|^2}{(D_v^2  b)(D_v f)- (D_v^2  f) (D_vb)}(s,x,m,\hv(s,x,m,p)),
\end{align*}
we shall explain its well-posedness without exploding to infinity in the following. Substituting this last expression into \eqref{rk_non_convex_H_1}, we can compute that
\begin{align*}
    D_{p}^2 H(s,x,m,p)=&\ \frac{|D_v b|^2}{\frac{(D_v b)(D_v^2 b)(D_v f)}{|D_v b|^2} - D_v^2  f}(s,x,m,\hv(s,x,m,p)).
\end{align*}
From Assmption (A1'), we know that $\left|\frac{(D_v b)(D_v^2  b)(D_v f)}{|D_v b|^2}\right|\le \frac{L^2 L_b^v }{\lambda_b}$, and together with the convexity of $f$ in $v$, the inequality $\lambda_v>\frac{2L^2L_b^v}{\lambda_b}$ yields that $D_{p}^2 H(s,x,m,p,q)<0$, which implies that $H$ is concave in $p$. 

In the HJB-FP system \eqref{HJB-FP}, if we define 
\begin{align*}
	U(t,x,\mu):=v^{t,\mu}(t,x),\quad (t,x,\mu)\in[0,T]\times\br^n\times\pr_2(\brn),
\end{align*}
then, from System \eqref{HJB-FP}, we can immediately see that it satisfies the following MFG master equation:
\begin{equation}\label{master}
	\left\{
	\begin{aligned}
		&\dd_t U(t,x,\mu)+H\left(t,x,\mu,D_x U(t,x,\mu),\frac{1}{2}D_x^2 U(t,x,\mu)\sigma(t,x,\mu)\right)\\
		&\quad\qquad\qquad +\int_\brn \bigg[D_p H \left(t,y,\mu,D_x U(t,y,\mu)\right)^\top D_y\frac{dU}{d\nu}(t,x,\mu)({y})\\
		&\qquad\qquad\qquad\qquad +\frac{1}{2}\text{Tr}\left(\left(\sigma\sigma^\top\right) (t,{y},\mu) D_y^2\frac{dU}{d\nu}(t,x,\mu)({y})\right)\bigg] d\mu(y)=0, \quad t\in[0,T),\\
		&U(T,x,\mu)=g(x,\mu),\quad \forall(x,\mu)\in\brn\times\pr_2(\brn).
	\end{aligned}	
	\right.
\end{equation}
The solution of the master equation \eqref{master} is a decoupling field of the mean field game HJB-FP system \eqref{HJB-FP}:
\begin{align*}
	v^{t,\mu}(s,x):=U\left(s,x,m^{t,\mu}(s)\right),\ (s,x)\in[t,T]\times\brn;
\end{align*}
we also refer to \cite{AB11,AB9',HZ2} for more discussions on the relations between the mean field game master equation and the HJB-FP system. The master equation \eqref{master} is widely studied in the existing literature, see \cite{GW,Anti_mono} for an analytical method for well-posedness when $\sigma$ is constant, and also see our previous work \cite{AB11} for a probabilistic approach when $b$ is linear. From \cite{AB11,AB9'}, we see that classical solution of the master equation \eqref{master} or the HJB-FP equations \eqref{HJB-FP} requires more restrictive assumptions and higher regularity on coefficients than just solution of the MFG \eqref{intro_MFG}. Therefore, the stochastic control method in the paper allows us to include more cases, and also to impose less regularity condition on the coefficients. So as to guarantee the solvability of the original mean field problem.

\subsection{Mean field type control problems with generic drifts}

We next study the solvability of the MFTC problem \eqref{intro_MFTC} without Assumption (B1); instead, we take the following assumptions on $b$ and $\sigma$:

\textbf{(B1')} The coefficient $b$ grows linearly and is $L$-Lipschitz continuous in $(x,v,m)$; it is continuously differentiable in $x$, $v$ and $m$, with the derivatives $D_x b$, $D_x b$ and $D_y\frac{db}{d\nu}$ being bounded by $L$, and they are $L$-Lipschitz continuous in $(x,v,m)$. Moreover, for any $s\in[0,T]$, $x,x'\in\brn$, $v,v'\in\brd$ and square-integrable random variables $\xi$ and $\xi'$ on the same probability space,
\begin{align*}
    &\left|b(s,\cdot)\Big|^{(x', \lr(\xi'), v')}_{(x, \lr(\xi), v)} -\left[(D_x, D_v) b(s,x,\lr(\xi),v)\right] 
    \begin{pmatrix}  x'-x\\ v'-v\end {pmatrix}
      -\e\left[D_y\frac{db}{d\nu}(s,x,\lr(\xi),v)(\xi)(\xi'-\xi)\right]\right|\\[2mm]
    \le\ & \frac{L_b^x|x'-x|^2+L_b^v|v'-v|^2+L_b^m \|\xi'-\xi\|_2^2}{1+ |x|\vee |x'|+ |v|\vee |v'|+\|\xi\|_2\vee \|\xi'\|_2},
\end{align*}
and there exists $\lambda_b >0$, such that
\begin{align*}
    (D_v b)(D_v b)^\top(s,x,m,v)\ \geq\  \lambda_b I_n, \quad \forall (s,x,v)\in [0,T]\times \brn\times \brd.
\end{align*}
The coefficient $\sigma$ is linear in $(x,v)$ as $\sigma(s,x,m)=\sigma_0(s)+\sigma_1(s)x+\sigma_3(s)\int_\brn ym(dy)$, with the norms of matrices $\sigma_0$, $\sigma_1$ and $\sigma_3)$ being bounded  by $L$. 

\begin{remark}
    Assumption (B1') is different from Assumption (A1'). This is because the state process of the MFTC problem \eqref{intro_MFTC} depends simultaneously on the distribution of the current controlled state, while the state process of MFG \eqref{intro_MFG} depends on the the population equilibrium distribution. In Assumption (B1'), when the coefficient $b$ is linear in $(x, m, v)$, then $L_b^x=L_b^v=L_b^m=0$. Therefore, (B1') is a nonlinear  extension of  the linear assumption on $b$ in Assumption (B1).
\end{remark}

We also need the following strong convexity on $f$: 

\textbf{(B3')} The functional $g$ satisfies the convexity \eqref{convex_g_mftc}, and there exist $\lambda_v> 0$ and $\lambda_x,\lambda_m\geq 0$, such that for any square-integrable random variables $\xi$ and $\xi'$ on the same probability space,
\begin{align}
	&f\left(s,x',\lr(\xi'),v'\right)-f(s,x,\lr(\xi),v) \notag \\
    \geq \ & \left[\begin{pmatrix}
        D_x f \\ D_v f
    \end{pmatrix}(s,x,\lr(\xi),v) \right]^\top \begin{pmatrix}
        x'-x\\ v'-v
    \end{pmatrix} +\e\left[\left(D_y\frac{d f}{d\nu}(s,x,\lr(\xi),v)(\xi)\right)^\top \left(\xi'-\xi\right)\right] \notag \\
	&+\lambda_v \left|v'-v\right|^2+\lambda_x|x'-x|^2+\lambda_m \|\xi'-\xi\|_2^2 . \label{convex:mgtc:generic}
\end{align}

In view of Assumption (B1'), the system of FBSDEs associated with the MFTC problem \eqref{intro_MFTC} are as follows:
\small
\begin{equation}\label{FB:mftc_generic}
\left\{
    \begin{aligned}
        &X_s = \xi+\int_t^s b(r,X_r,\lr(X_r),v_r)dr+ \int_t^s \left\{\sigma_0(r)+\sigma_1(r)X_r+\sigma_3(r)\e[X_r]\right\}dB_r,\\
        &P_s = -\int_s^T Q_r dB_r+D_x g(X_T,\lr(X_T))D_xg\left(X_T,\lr(X_T)\right) +\widetilde{\e}\left[D_y\frac{d g}{d\nu}\left(\widetilde{X_T},\lr(X_T)\right)(X_T)\right]\\
        &\ \qquad +\int_s^T \bigg\{D_x b(r,X_r,\lr(X_r),v_r)^\top P_r+\sum_{j=1}^n \left(\sigma^j_1(r)\right)^\top Q_r^j + D_x f(r,X_r,\lr(X_r),v_r)\\
        &\qquad\qquad\qquad + \widetilde{\e}\left[\widetilde{{P}_r}^\top\left(D_y \frac{d b}{d\nu}\left(r,\widetilde{{X}_r},\lr(X_r),\widetilde{v_r}\right)\left(X_r\right) \right) \right]+\sum_{j=1}^n \left(\sigma_3^j(r)\right)^\top \e\left[Q^j_r\right]\\
        &\qquad\qquad\qquad +\widetilde{\e}\left[D_y \frac{d f}{d\nu}\left(r,\widetilde{X_r},\lr(X_r),\widetilde{v_r}\right)(X_r) \right]\bigg\}dr,\quad s\in[t,T],
    \end{aligned}
\right.
\end{equation}
\normalsize
with the condition
\begin{equation}\label{FB:mftc_generic_condition}
    D_v b(s,X_s,\lr(X_s),v_s)^\top P_s+ D_v f(s,X_s,\lr(X_s),v_s)=0,\quad s\in[t,T].
\end{equation}
We now give the corresponding maximum principle for MFTC \eqref{intro_MFTC} with a generic $b$.

\begin{theorem}\label{lem:MP2'}
	Under Assumptions (B1'), (B2) and (B3'), suppose that $\lambda_x+\lambda_m\geq\frac{L^2(L_b^x+L_b^m)}{\lambda_b}$ and $\lambda_v>\frac{L^2 L_b^v}{\lambda_b}$. If FBSDEs \eqref{FB:mftc_generic}-\eqref{FB:mftc_generic_condition} have a solution $\left(\bar{X},\bar{P},\bar{Q},\bar{v}\right)\in \sr^2_\f(t,T)\times\sr^2_{\f}(t,T)\times\left(\lr^2_{\f}(t,T)\right)^n\times \lr^2_{\f}(t,T)$, then, $\bar{v}$ is the unique optimal control of the MFTC problem \eqref{intro_MFTC}.
\end{theorem}

\begin{proof}
Similar to the proof of Theorem~\ref{lem:MP1'}, we estimate the difference  $J(v)-J\left(\bar{v}\right)$ for any control $v\in \lr_{\f}^2(t,T)$. From Assumption (B3'), denoting by $X^v$ the state process corresponding to a control $v$, we have
\begin{align}
    &J(v)-J\left(\bar{v}\right) \notag \\
    \geq\  & \e\Bigg\{\int_t^T \Bigg[ \left[\begin{pmatrix}
        D_x f\\ D_v f
    \end{pmatrix} \left(s,\bar{X}_s,\lr\left(\bar{X}_s\right),\bar{v}_s\right)\right]^\top \begin{pmatrix} X^v_s-\bar{X}_s\\
    v_s-\bar{v}_s\end{pmatrix} + \lambda_v \left|v_s-\bar{v}_s\right|^2+ \lambda_x\left|X^v_s-\bar{X}_s\right|^2 \notag\\
    &\qquad\qquad+\widetilde{\e}\left[\left(D_y\frac{d f}{d\nu}\left(s,\bar{X}_s,\lr(\bar{X}_s),\bar{v}_s\right)\left(\widetilde{\bar{X}_s}\right)\right)^\top \left(\widetilde{X^v_s}-\widetilde{\bar{X}_s}\right)\right]+\lambda_m \left\|\widetilde{X^v_s}-\widetilde{\bar{X}_s}\right\|_2^2 
     \Bigg]ds \notag \\
    &\quad +D_x g\left(\bar{X}_T\right)^\top \left(X^v_T-\bar{X}_T\right)+ \widetilde{\e}\left[\left(D_y\frac{d g}{d\nu}\left(\bar{X}_T,\lr(\bar{X}_T)\right)\left(\widetilde{\bar{X}_T}\right)\right)^\top \left(\widetilde{X^v_T}-\widetilde{\bar{X}_T}\right)\right]\Bigg\}. \label{thm:ger_max_5'}
\end{align}
By applying It\^o's formula to the product $\bar{P}_s^\top \left(X^v_s-\bar{X}_s\right)$ and then taking expectation, we have
\begin{align}
    &\e\left[ D_x g\left(\bar{X}_T\right)^\top \left(X^v_T-\bar{X}_T\right)\right]+\e\widetilde{\e}\left[\left(D_y\frac{d g}{d\nu}\left(\widetilde{\bar{X}_T},\lr(\bar{X}_T)\right)\left(\bar{X}_T\right)\right)^\top \left(X^v_T-\bar{X}_T\right)\right] \notag \\
    =\ & \e\Bigg\{\int_t^T \bar{P}_s^\top \left[b (s, \cdot)\Big|^{\left( X^v_s,\lr\left(X^v_s\right),v_s\right)}_{\left(\bar{X}_s,\lr\left(\bar{X}_s\right),\bar{v}_s\right)} \right] + \sum_{j=1}^n \left(\bar{Q}^j_s\right)^\top\sigma_3^j(s)\widetilde{\e}\left[\widetilde{X^v_s}-\widetilde{\bar{X}_s}\right]  \notag \\
    &\quad\qquad -\left[ D_x b\left(s,\bar{X}_s,\lr\left(\bar{X}_s\right),\bar{v}_s\right)^\top \bar{P}_s+D_x f\left(s,\bar{X}_s,\lr\left(\bar{X}_s\right),\bar{v}_s\right) \right]^\top \left( X^v_s-\bar{X}_s\right) \notag \\
    &\quad\qquad - \widetilde{\e}\left[\widetilde{\bar{P}_s}^\top\left(D_y \frac{d b}{d\nu}\left(s,\widetilde{\bar{X}_s},\lr(\bar{X}_s),\widetilde{\bar{v}_s}\right)\left(\bar{X}_s\right) \right) \left( X^v_s-\bar{X}_s\right)\right]-\sum_{j=1}^n \widetilde{\e} \left[\widetilde{\bar{Q}^j_s}^\top \sigma_3^j(s) \left( X^v_s-\bar{X}_s\right)\right]\notag\\
    &\quad\qquad - \widetilde{\e}\left[\left(D_y \frac{d f}{d\nu}\left(s,\widetilde{\bar{X}_s},\lr(\bar{X}_s),\widetilde{\bar{v}_s}\right)\left(\bar{X}_s\right) \right)^\top \left( X^v_s-\bar{X}_s\right)\right] ds\Bigg\}. \notag
\end{align}
By Fubini's lemma, we know that 
\small
\begin{align*}
    &(i)\ \ \;\e\widetilde{\e}\left[\left(D_y\frac{d g}{d\nu}\left(\widetilde{\bar{X}_T},\lr(\bar{X}_T)\right)\left(\bar{X}_T\right)\right)^\top \left(X^v_T-\bar{X}_T\right)\right]=\e\widetilde{\e}\left[\left(D_y\frac{d g}{d\nu}\left({\bar{X}_T},\lr(\bar{X}_T)\right)\left(\widetilde{\bar{X}_T}\right)\right)^\top \left(\widetilde{X^v_T}-\widetilde{\bar{X}_T}\right)\right], \\
    &(ii)\ \;\e\widetilde{\e}\left[\widetilde{\bar{P}_s}^\top\left(D_y \frac{d b}{d\nu}\left(s,\widetilde{\bar{X}_s},\lr(\bar{X}_s),\widetilde{\bar{v}_s}\right)\left(\bar{X}_s\right) \right) \left( X^v_s-\bar{X}_s\right)\right]=\e\widetilde{\e}\left[\bar{P}_s^\top\left(D_y \frac{d b}{d\nu}\left(s,\bar{X}_s,\lr(\bar{X}_s),\bar{v}_s\right)\left(\widetilde{\bar{X}_s}\right) \right) \left( \widetilde{X^v_s}-\widetilde{\bar{X}_s}\right)\right],\\
    &(iii)\ \e\widetilde{\e} \left[\widetilde{\bar{Q}^j_s}^\top \sigma_3^j(s) \left( X^v_s-\bar{X}_s\right)\right]=\e\widetilde{\e} \left[{\bar{Q}^j_s}^\top \sigma_3^j(s) \left( \widetilde{X^v_s}-\widetilde{\bar{X}_s}\right)\right],\\
    &(iv)\ \;\e\widetilde{\e}\left[\left(D_y \frac{d f}{d\nu}\left(s,\widetilde{\bar{X}_s},\lr(\bar{X}_s),\widetilde{\bar{v}_s}\right)\left(\bar{X}_s\right) \right)^\top \left( X^v_s-\bar{X}_s\right)\right]=\e\widetilde{\e}\left[\left(D_y \frac{d f}{d\nu}\left(s,\bar{X}_s,\lr(\bar{X}_s),\bar{v}_s\right)\left(\widetilde{\bar{X}_s}\right) \right)^\top \left(\widetilde{X^v_s}-\widetilde{\bar{X}_s}\right)\right].
\end{align*}
\normalsize
Therefore, we know that 
\begin{align}
    &\e\left\{ D_x g\left(\bar{X}_T\right)^\top \left(X^v_T-\bar{X}_T\right)+\widetilde{\e}\left[\left(D_y\frac{d g}{d\nu}\left({\bar{X}_T},\lr(\bar{X}_T)\right)\left(\widetilde{\bar{X}_T}\right)\right)^\top \left(\widetilde{X^v_T}-\widetilde{\bar{X}_T}\right)\right] \right\} \notag \\
    =\ & \e\Bigg\{\int_t^T \bar{P}_s^\top \bigg[b(s, \cdot) \Big|^{\left(X^v_s,\lr\left(X^v_s\right),v_s\right)}_{\left(\bar{X}_s,\lr\left(\bar{X}_s\right),\bar{v}_s\right)}  -D_x b\left(s,\bar{X}_s,\lr\left(\bar{X}_s\right),\bar{v}_s\right) \left( X^v_s-\bar{X}_s\right)  \notag \\
    &\quad\qquad\qquad - \widetilde{\e}\left[\bar{P}_s^\top\left(D_y \frac{d b}{d\nu}\left(s,\bar{X}_s,\lr(\bar{X}_s),\bar{v}_s\right)\left(\widetilde{\bar{X}_s}\right) \right) \left( \widetilde{X^v_s}-\widetilde{\bar{X}_s}\right)\right] \bigg]\notag\\
    &\quad\qquad -D_x f\left(s,\bar{X}_s,\lr\left(\bar{X}_s\right),\bar{v}_s\right)^\top \left( X^v_s-\bar{X}_s\right) \notag \\
    &\quad\qquad - \widetilde{\e}\left[\left(D_y \frac{d f}{d\nu}\left(s,\bar{X}_s,\lr(\bar{X}_s),\bar{v}_s\right)\left(\widetilde{\bar{X}_s}\right) \right)^\top \left(\widetilde{X^v_s}-\widetilde{\bar{X}_s}\right)\right] ds\Bigg\}. \notag
\end{align}
From \eqref{FB:mftc_generic_condition}, we know that 
\begin{align}
    &\e\Bigg\{ D_x g\left(\bar{X}_T\right)^\top \left(X^v_T-\bar{X}_T\right)+\widetilde{\e}\left[\left(D_y\frac{d g}{d\nu}\left({\bar{X}_T},\lr(\bar{X}_T)\right)\left(\widetilde{\bar{X}_T}\right)\right)^\top \left(\widetilde{X^v_T}-\widetilde{\bar{X}_T}\right)\right] \notag\\
    &\quad +\int_t^T \left[\begin{pmatrix}
        D_x f\\ D_v f
    \end{pmatrix}\left(s,\bar{X}_s,\lr\left(\bar{X}_s\right),\bar{v}_s\right)\right]^\top \begin{pmatrix} X^v_s-\bar{X}_s\\ v_s-\bar{v}_s\end{pmatrix} \notag\\
    &\qquad\qquad + \widetilde{\e}\left[\left(D_y \frac{d f}{d\nu}\left(s,\bar{X}_s,\lr(\bar{X}_s),\bar{v}_s\right)\left(\widetilde{\bar{X}_s}\right) \right)^\top \left(\widetilde{X^v_s}-\widetilde{\bar{X}_s}\right)\right] ds\Bigg\} \notag \\
    =\ & \e\Bigg\{\int_t^T \bar{P}_s^\top \bigg[b (s, \cdot)\Big|^{\left(X^v_s,\lr\left(X^v_s\right),v_s\right)}_{\left(\bar{X}_s,\lr\left(\bar{X}_s\right),\bar{v}_s\right)}  -\left[(D_x b, D_v b) \left(s,\bar{X}_s,\lr\left(\bar{X}_s\right),\bar{v}_s\right)\right] \begin{pmatrix} X^v_s-\bar{X}_s\\ v_s-\bar{v}_s\end{pmatrix} \notag\\
    &\quad\qquad\qquad -\widetilde{\e}\left[\left(D_y \frac{d b}{d\nu}\left(s,\bar{X}_s,\lr(\bar{X}_s),\bar{v}_s\right)\left(\widetilde{\bar{X}_s}\right) \right) \left( \widetilde{X^v_s}-\widetilde{\bar{X}_s}\right)\right] \bigg] ds\Bigg\}. \label{thm:ger_max_4'}
\end{align}
In view of \eqref{FB:mftc_generic_condition} and Assumption (B1'), we can compute that
\begin{align*}
    \bar{P}_s=-\left[\left((D_vb)(D_vb)^\top\right)^{-1}(D_vb)\left(s,\bar{X}_s,\lr\left(\bar{X}_s\right),\bar{v}_s\right)\right](D_v f)\left(s,\bar{X}_s,\lr\left(\bar{X}_s\right),\bar{v}_s\right),
\end{align*}
and then, from Assumption (B2), we have the following cone property:
\begin{equation}\label{p_norm'}
\begin{split}
    \left|\bar{P}_s\right|\le \frac{L^2}{\lambda_b}\left(1+\left|\bar{X}_s\right|+W_2\left(\lr(\bar{X}_s),\delta_0\right)+\left|\bar{v}_s\right|\right).
\end{split}
\end{equation}
Therefore, we deduce from Assumption (B1') that 
\begin{align}
    &\Bigg|\e\Bigg\{\int_t^T \bar{P}_s^\top \bigg[b (s, \cdot)\Big|^{ \left(X^v_s,\lr\left(X^v_s\right),v_s\right)}_{\left(\bar{X}_s,\lr\left(\bar{X}_s\right),\bar{v}_s\right)} -\left[(D_x b, D_v b) \left(s,\bar{X}_s,\lr\left(\bar{X}_s\right),\bar{v}_s\right)\right] \begin{pmatrix} X^v_s-\bar{X}_s\\ v_s-\bar{v}_s\end{pmatrix} \notag\\
    &\qquad\qquad\qquad -\widetilde{\e}\left[\bar{P}_s^\top\left(D_y \frac{d b}{d\nu}\left(s,\bar{X}_s,\lr(\bar{X}_s),\bar{v}_s\right)\left(\widetilde{\bar{X}_s}\right) \right) \left( \widetilde{X^v_s}-\widetilde{\bar{X}_s}\right)\right] \bigg] ds\Bigg\} \Bigg| \notag\\
    \le\ & \frac{L^2}{\lambda_b} \e\left\{\int_t^T \left[ L_b^x \left|X^v_s-\bar{X}_s\right|^2 +L_b^v \left|v_s-\bar{v}_s\right|^2+L_b^m \left\|\widetilde{X^v_s}-\widetilde{\bar{X}_s}\right\|_2^2 \right] ds  \right\}. \label{thm:ger_max_6'}
\end{align}
Substituting \eqref{thm:ger_max_6'} into \eqref{thm:ger_max_4'}, we know that
\begin{align*}
    &\e\Bigg\{ D_x g\left(\bar{X}_T\right)^\top \left(X^v_T-\bar{X}_T\right)+\widetilde{\e}\left[\left(D_y\frac{d g}{d\nu}\left({\bar{X}_T},\lr(\bar{X}_T)\right)\left(\widetilde{\bar{X}_T}\right)\right)^\top \left(\widetilde{X^v_T}-\widetilde{\bar{X}_T}\right)\right] \\
    &\quad +\int_t^T \left[\begin{pmatrix}
        D_x f\\ D_v f
    \end{pmatrix} \left(s,\bar{X}_s,\lr\left(\bar{X}_s\right),\bar{v}_s\right)\right]^\top \begin{pmatrix} X^v_s-\bar{X}_s\\ v_s-\bar{v}_s\end{pmatrix}  ds\\
    &\quad + \int_t^T\widetilde{\e}\left[\left(D_y \frac{d f}{d\nu}\left(s,\bar{X}_s,\lr(\bar{X}_s),\bar{v}_s\right)\left(\widetilde{\bar{X}_s}\right) \right)^\top \left(\widetilde{X^v_s}-\widetilde{\bar{X}_s}\right)\right] ds\Bigg\}  \\
    \geq\ & -\frac{L^2}{\lambda_b} \e\left[\int_t^T (L_b^x+L_b^m) \left|X^v_s-\bar{X}_s\right|^2 +L_b^v \left|v_s-\bar{v}_s\right|^2 ds  \right].
\end{align*}
Combining the last inequality with \eqref{thm:ger_max_5'}, we know that 
\begin{align*}
    J(v)-J\left(\bar{v}\right) \geq\  & \e\left\{\int_t^T \left[ \left(\lambda_x+\lambda_m-\frac{L^2(L_b^x+L_b^m)}{\lambda_b}\right) \left|X^v_s-\bar{X}_s\right|^2+ \left(\lambda_v-\frac{L^2L_b^v}{\lambda_b}\right)\left|v_s-\bar{v}_s\right|^2\right]ds\right\}\\
    \geq\ & \left(\lambda_v-\frac{L^2L_b^v}{\lambda_b}\right) \e\left[\int_t^T \left|v_s-\bar{v}_s\right|^2ds\right],
\end{align*}
from which we obtain the desired result. 
\end{proof}

We next show that Assumption (B1') and the strong convexity condition in (B3') yield that the coefficients of FBSDEs \eqref{FB:mftc_generic} satisfy the $\beta$-monotonicity, which consequently implies the global well-posedness of FBSDEs \eqref{FB:mftc_generic}.

\begin{theorem}\label{thm:mono_generic'}
    Under Assumptions (B1'), (B2) and (B3'), suppose that $\lambda_x+\lambda_m\geq\frac{L^2(L_b^x+L_b^m)}{\lambda_b}$ and $\lambda_v>\frac{L^2 L_b^v}{\lambda_b}$, then FBSDEs \eqref{FB:mftc_generic} have a unique global solution.
\end{theorem}

\begin{proof}
We shall show that the coefficients of FBSDEs \eqref{FB:mftc_generic} satisfy Condition~\ref{Condition_mono}. For $s\in[t,T]$, $X,X',P,P'\in L^2\left(\Omega,\f,\mathbb{P};\brn\right)$, $Q,Q'\in  L^2\left(\Omega,\f,\mathbb{P};\br^{n\times n}\right)$ and $V,V'\in L^2\left(\Omega,\f,\mathbb{P};\brd\right)$ satisfying
\begin{equation}\label{thm_generic_6}
\begin{aligned}
    &D_v b(s,X,\lr(X),V)^\top P+ D_v f(s,X,\lr(X),V)=0,\\
    &D_v b(s,X',\lr(X'),V')^\top P+ D_v f(s,X',\lr(X'),V')=0,
\end{aligned}
\end{equation}
and by the Fubini's lemma, we can compute that
\begin{align}
	&\e\bigg\{\left[b(s,X',\lr(X'),V')-b(s,X,\lr(X),V)\right]^\top  (P'-P) \notag\\
    &\quad +\sum_{j=1}^n \left[\sigma^j_1(s)(X'-X)+\sigma^j_3(s)\e[X'-X]\right]^\top  \left({Q'}^{j}-Q^j\right) \notag \\
    &\quad -\bigg\{D_x b(s,X',\lr(X'),V')^\top P' -D_x b(s,X,\lr(X),V)^\top P \notag\\
    &\quad\qquad + \widetilde{\e}\left[\widetilde{{P'}}^\top\left(D_y \frac{d b}{d\nu}\left(s,\widetilde{X'},\lr(X'),\widetilde{V'}\right)\left(X'\right) \right) \right]- \widetilde{\e}\left[\widetilde{{P}}^\top\left(D_y \frac{d b}{d\nu}\left(s,\widetilde{X},\lr(X),\widetilde{V}\right)\left(X\right) \right) \right]\notag \\
    &\quad\qquad +\sum_{j=1}^n \left(\sigma^j_1(s)\right)^\top \left({Q'}^j-Q^j\right)+\sum_{j=1}^n \left(\sigma_3^j(s)\right)^\top \e\left[{Q'}^j-Q^j\right]  \notag\\
    &\quad\qquad + D_x f(s,X',\lr(X'),V')-D_x f(s,X,\lr(X),V) \notag \\
    &\quad\qquad +\widetilde{\e}\left[D_y \frac{d f}{d\nu}\left(s,\widetilde{X'},\lr(X'),\widetilde{V'}\right)(X') \right]-\widetilde{\e}\left[D_y \frac{d f}{d\nu}\left(s,\widetilde{X},\lr(X),\widetilde{V}\right)(X) \right]\bigg\}^\top (X'-X)\bigg\} \notag \\
	=\ & \e\bigg\{ -\bigg[b(s, \cdot) \Big|^{\left(X,\lr(X),V\right)}_{\left(X',\lr(X'),V'\right)} - D_x b\left(s,X',\lr(X'),V'\right) (X-X') \notag \\
	&\quad\qquad - \widetilde{\e}\left[D_y \frac{d b}{d\nu}\left(s,\widetilde{X'},\lr(X'),\widetilde{V'}\right)(X')\right](X-X') \bigg]^\top P' \notag \\
    &\ \quad -\bigg[b(s, \cdot) \Big|_{\left(X,\lr(X),V\right)}^{\left(X',\lr(X'),V'\right)} - D_x b\left(s,X,\lr(X),V\right) (X'-X) \notag\\
    &\quad\qquad - \widetilde{\e}\left[D_y \frac{d b}{d\nu}\left(s,\widetilde{X},\lr(X),\widetilde{V}\right)(X)\right](X'-X)\bigg]^\top P \notag \\
    &\ \quad - \left(D_xf(s,X',\lr(X'),V')-D_xf(s,X,\lr(X),V)\right)^\top  (X'-X) \notag\\
    &\ \quad - \widetilde{\e}\left[D_y \frac{d f}{d\nu}\left(s,\widetilde{X'},\lr(X'),\widetilde{V'}\right)(X')-D_y \frac{d f}{d\nu}\left(s,\widetilde{X},\lr(X),\widetilde{V}\right)(X) \right]^\top (X'-X)\bigg\} \notag\\
    =\ & \e\Bigg\{ -\bigg[b(s, \cdot) \Big|^{\left(X,\lr(X),V\right)}_{\left(X',\lr(X'),V'\right)} - \left[(D_x b, D_v b) \left(s,X',\lr(X'),V'\right)\right]\begin{pmatrix} X-X'\\ V-V'\end{pmatrix}\notag \\
    &\quad\qquad - \widetilde{\e}\left[D_y \frac{d b}{d\nu}\left(s,X',\lr(X'),V'\right)\left(\widetilde{X'}\right)\left(\widetilde{X}-\widetilde{X'}\right)\right] \bigg]^\top P' \notag \\
    &\ \quad -\bigg[b(s, \cdot) \Big|_{\left(X,\lr(X),V\right)}^{\left(X',\lr(X'),V'\right)} - \left[(D_x b, D_v b) \left(s,X,\lr(X),V\right)\right] \begin{pmatrix} X'-X \\
    V'-V\end{pmatrix} \notag\\ 
    &\quad\qquad - \widetilde{\e}\left[D_y \frac{d b}{d\nu}\left(s,X,\lr(X),V\right)\left(\widetilde{X}\right)\left(\widetilde{X'}-\widetilde{X}\right)\right]\bigg]^\top P \notag \\
    &\ \quad - \left[\begin{pmatrix}
        D_x f \\ D_v f
    \end{pmatrix}(s,\cdot) \bigg|^{\left(X',\lr(X'),V'\right)}_{\left(X,\lr(X),V\right)} \right]^\top \begin{pmatrix}
        X'-X \\ V'-V
    \end{pmatrix} \notag\\
    &\ \quad - \widetilde{\e}\left[\left(D_y \frac{d f}{d\nu}\left(s,X',\lr(X'),V'\right)\left(\widetilde{X'}\right)-D_y \frac{d f}{d\nu}\left(s,X,\lr(X),V\right)\left(\widetilde{X}\right) \right)^\top \left(\widetilde{X'}-\widetilde{X}\right)\right]\Bigg\}. \label{thm_generic_1'}
\end{align} 
Similar to \eqref{p_norm'}, from \eqref{thm_generic_6}, Assumptions (B1') and (B2), we deduce that
\begin{align*}
    |P'|\le \frac{L^2}{\lambda_b}\left(1+|X'|+W_2(\lr(X'),\delta_0)+\left|V'\right|\right).
\end{align*}
Therefore, from Assumption (B1'), we arrive with the following
\begin{align}
    &\e\Bigg\{ -\bigg[b\left(s,\cdot\right)\Big|^{\left(X,\lr(X),V\right)}_{\left(X',\lr(X'),V'\right)} - \left[(D_x b, D_v b) \left(s, X',\lr\left(X'\right),V'\right)\right] \begin{pmatrix} X-X' \\ V-V'\end{pmatrix} \notag \\
    &\quad\qquad - \widetilde{\e}\left[D_y \frac{d b}{d\nu}\left(s,\widetilde{X'},\lr(X'),\widetilde{V'}\right)(X')\right](X-X') \bigg]^\top P' \Bigg\}\notag \\
    \le\ & \frac{L^2}{\lambda_b}\ \e\left[ (L_b^x+L_b^m) |X'-X|^2+L_b^v |V'-V|^2 \right]. \label{thm_generic_2'}
\end{align}
In a  similar fashion, we also have
\begin{align}
    &\e\Bigg\{ -\bigg[b(s, \cdot) \Big|_{\left(X,\lr(X),V\right)}^{\left(X',\lr(X'),V'\right)} - \left[(D_x b, D_v b) \left(s,X,\lr(X),V\right) \right] \begin{pmatrix}  X'-X\\ V'-V \end{pmatrix}  \notag \\
    &\quad\qquad - \widetilde{\e}\left[D_y \frac{d b}{d\nu}\left(s,\widetilde{X},\lr(X),\widetilde{V}\right)(X)\right](X'-X) \bigg]^\top P \Bigg\}\notag \\
    \le\ & \frac{L^2}{\lambda_b}\ \e\left[ (L_b^x+L_b^m) |X'-X|^2+L_b^v |V'-V|^2 \right]. \label{thm_generic_3'}
\end{align}
Since $f$ satisfies Assumption (B3'),  we know that
\begin{align}
    &\e\Bigg\{ - \left[\begin{pmatrix}
        D_x f \\ D_v f
    \end{pmatrix}(s,\cdot) \bigg|^{\left(X',\lr(X'),V'\right)}_{\left(X,\lr(X),V\right)} \right]^\top \begin{pmatrix}
        X'-X \\ V'-V
    \end{pmatrix}  \notag\\
    &\ \quad - \widetilde{\e}\left[\left(D_y \frac{d f}{d\nu}\left(s,X',\lr(X'),V'\right)\left(\widetilde{X'}\right)-D_y \frac{d f}{d\nu}\left(s,X,\lr(X),V\right)\left(\widetilde{X}\right) \right)^\top \left(\widetilde{X'}-\widetilde{X}\right)\right]\Bigg\} \notag\\
    \le\ & -2\lambda_v \left\| V'-V\right\|_2^2 -2\left(\lambda_x+\lambda_m\right) \left\| X'-X \right\|_2^2.\label{thm_generic_4'}
\end{align}
Substituting \eqref{thm_generic_2'}-\eqref{thm_generic_4'} back into \eqref{thm_generic_1'}, we conclude with the following:
\begin{align*}
	&\textit{The left hand side of \eqref{thm_generic_1'}}\\
	\le\ & -2\left(\lambda_v-\frac{L^2L_b^v}{\lambda_b}\right) \left\|V'-V\right\|_2^2 -  2\left(\lambda_x+\lambda_m-\frac{L^2(L_b^x+L_b^m)}{\lambda_b}\right) \left\|X'-X\right\|_2^2 ,
\end{align*}  
from which we know that Condition \ref{Condition_mono} (i)(a) is valid with $\Lambda_\beta=2\lambda_v-\frac{2L^2L_b^v}{\lambda_b}$ and $\Gamma_\beta=0$. In a similar way, we can show that Condition \ref{Condition_mono} (i)(b) is also satisfied. Therefore, as a consequence of Lemma~\ref{lem:1}, we obtain the well-posedness of FBSDEs \eqref{FB:mftc_generic}.
\end{proof}

As an immediate consequence of Theorems~\ref{lem:MP2'} and \ref{thm:mono_generic'}, we obtain the following solvability of the MFTC problem \eqref{intro_MFTC} with a generic $b$.

\begin{corollary}
    Under assumptions in Theorem~\ref{thm:mono_generic'}, the MFTC problem \eqref{intro_MFTC} has a unique optimal control. 
\end{corollary}

\section*{\normalsize Acknowledgement}

Alain Bensoussan is supported by the National Science Foundation under grant NSF-DMS-2204795. Ziyu Huang acknowledges the financial supports as a postdoctoral fellow from Department of Statistics of The Chinese University of Hong Kong. Shanjian Tang is supported by the National Natural Science Foundation of China under grant nos. 12031009 and 11631004. Phillip Yam acknowledges the financial supports from HKGRF-14301321 with the project title ``General Theory for Infinite Dimensional Stochastic Control: Mean Field and Some Classical Problems'', and HKGRF-14300123 with the project title ``Well-posedness of Some Poisson-driven Mean Field Learning Models and their Applications''. He also thanks University of Texas at Dallas for the kind invitation to be a Visiting Professor in the Naveen Jindal School of Management during his sabbatical leave.

\end{document}